\documentclass[11pt]{elsarticle} 
\usepackage{t1enc,amsfonts,latexsym,amsmath,amsthm,verbatim,amssymb,graphics,mathdots,pstool,color}
\usepackage[utf8]{inputenc} 
\usepackage{pgfplots}
\usepackage{geometry}
\usepackage[showonlyrefs]{mathtools}
\usepackage{courier}
\usepackage{listings}
\usepackage{wrapfig}

\usepackage{ amsmath,amsfonts}
\usepackage{graphicx} 

\usepackage{mathrsfs}

\newcommand{\rank}{\mathsf{rank}}

\newcommand{\ran}{\mathsf{Ran}}
\renewcommand{\Re}{\mathsf{Re}}

\DeclareMathOperator*{\argmin}{arg\,min ~}

\newcommand{\fro}[1]{\left\| #1\right\|^2}
\newcommand{\scal}[1]{\left \langle #1 \right \rangle}
\newcommand{\para}[1]{\left( #1\right)}
\newcommand{\R}{\mathbb{R}}
\newcommand{\C}{\mathbb{C}}
\newcommand{\h}{\mathbb{H}}
\newcommand{\m}{\mathbb{M}}

\newtheorem{theorem}{Theorem}[section]

\newtheorem*{theoremextra1}{Theorem \ref{vonN}}
\newtheorem{proposition}[theorem]{Proposition}

\newtheorem{corollary}[theorem]{Corollary}

\def\B {{\mathcal B }}

\newcommand{\V}{\mathcal{V}}
\newcommand{\W}{\mathcal{W}}
\newcommand{\U}{\mathcal{U}}

\newcommand{\I}{\mathscr{I}}
\newcommand{\J}{\mathscr{J}}
\newcommand{\N}{\mathbb{N}}
\renewcommand{\L}{\mathcal{L}}
\renewcommand{\S}{\mathcal{S}}
\newcommand{\M}{\mathcal{M}}
\renewcommand{\H}{\mathcal{H}}

\newcommand{\prox}{\mathsf{prox}}

\newtheorem{ex}[theorem] {Example}

\begin{document}
\title{On convexification/optimization of functionals including an $l^2$-misfit term}
\author{Marcus Carlsson}
\ead{marcus.carlsson@math.lu.se}
\address{Centre for Mathematical Sciences, Lund University, Box 118, 22100 Lund, Sweden}

\begin{keyword}
Fenchel conjugate \sep convex envelope \sep non-convex/non-smooth optimization
\MSC[2010] 49M20 \sep 65K10 \sep 90C26
\end{keyword}

\begin{abstract}
We provide theory for computing the lower semi-continuous convex envelope of functionals of the type \begin{equation}\label{pg}f(x)+\frac{1}{2}\|x-d\|^2,\end{equation} and discuss applications to various non-convex optimization problems. The latter term is a data fit term whereas $f$ provides structural constraints on $x$. By minimizing \eqref{pg}, possibly with additional constraints, we thus find a tradeoff between matching the measured data and enforcing a particular structure on $x$, such as sparsity or low rank. For these particular cases, the theory provides alternatives to convex relaxation techniques such as $\ell^1$-minimization (for vectors) and nuclear norm-minimization (for matrices). For functionals of the form $$f(x)+\frac{1}{2}\|Ax-d\|^2,$$ where the convex envelope usually is not explicitly computable, we provide theory for how minimizers of (explicitly computable) approximations of the convex envelope relate to minimizers of the original functional. In particular, we give explicit conditions on when the two coincide.
\end{abstract}

\maketitle
\section{Introduction}
\textcolor{red}{The article \cite{carlsson2018convex} is a condensed and improved version of this article, in which we denote $\S^2_\gamma$ by $\mathcal{Q}_\gamma$ and call it the quadratic envelope. This article still contains much more information, especially regarding computational aspects. We will not update notation/terminology in this article. }

The purpose of this article is to convexify, or partially convexify, functionals of the type \begin{equation}\label{t5} \|x\|_0+\frac{1}{2}\|Ax-d\|^2_2\end{equation} where $x\in\R^n$, and \begin{equation}\label{f667} \rank(X)+\frac{1}{2}\|X-D\|^2_F\end{equation} where $X\in \m_{m,n}$ (the space of $m\times n$-matrices with the Frobenius norm). In other words, we are interested in computing the lower semi-continuous (abbreviated l.s.c.~) convex envelope or at least an approximation thereof. We will also consider weighted norms and penalty terms like \begin{equation}\label{f6}f(X)=\left\{\begin{array}{cc}
                     0 & \rank (X)\leq K,\\
                     \infty & \text{ else.}
                   \end{array}
\right.
\end{equation}
in order to treat problems where a matrix of a fixed rank is sought. Such functionals appear in a multitude of optimization problems, where the goal is to find a point $x$ such that the functional attains its minimum, possibly with additional constraints on $x$. We refer to the overview article \cite{tseng2010approximation} which includes a long list of applications. The problem of minimizing \eqref{t5} and \eqref{f667} differ significantly in that \eqref{f667} has a closed form solution whereas solving \eqref{t5} is NP-hard. However, minimization of \eqref{f667} over a subspace or in combination with additional priors, is also a hard well-known problem with many applications, and knowing the l.s.c. convex envelope can help to find approximate solutions, as we advocate in this paper. We refer to \cite{larsson2016convex,recht2010guaranteed} and the references therein for examples of applications.

Since the functional \eqref{f6}, as well as $\|\cdot\|_0$ and $\rank(\cdot)$, are non-convex, it is tempting to replace them by their convex envelopes. However, in all three cases the convex envelope equals 0. To obtain problems that are efficiently solvable, it is therefore popular to replace e.g. $\|\cdot\|_0$ with the $\ell^1$-norm or $\rank(\cdot)$ by the nuclear norm, a strategy which is sometimes called convex relaxation, thus obtaining a convex problem reminiscent of the original problem. Such methods have a long history, but has received new attention in recent times due to the realization that the original problem and the convex relaxation under certain assumptions have the same solution, as pioneered in the work concerning compressed sensing \cite{donoho2006most,candes2006robust}. The argument behind the choice of convex replacement is often that the functionals in question are the convex envelopes of the original ones when restricted to the unit ball, see e.g. \cite{recht2010guaranteed}.

\begin{wrapfigure}{r}{0.5\textwidth}
  \begin{center}
     \includegraphics[width=.5\textwidth]{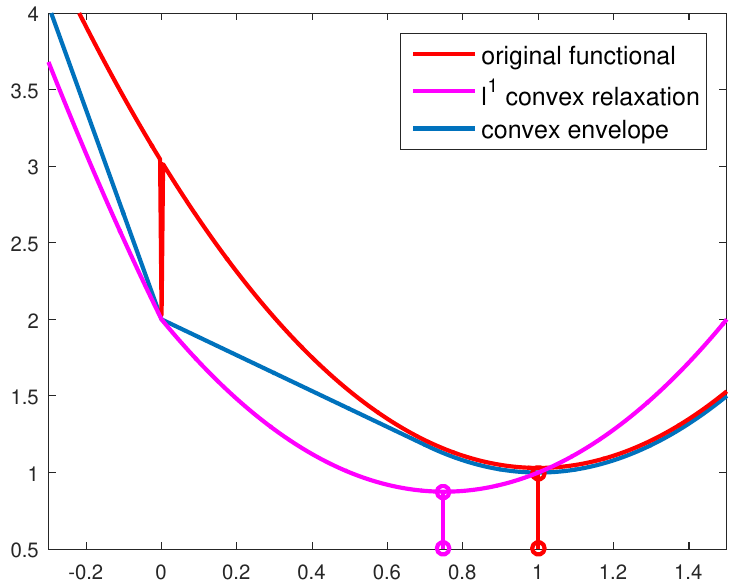}\\
  \end{center}
  \caption{Illustration of a non-convex, non-continuous functional together with its convex envelope and a ``traditional'' convex relaxation.}
  \label{figintro}
\end{wrapfigure}

Despite the success of these methods, there is a notable difference between the functional $\|x\|_0$ and $\|x\|_1$ for large values of $x$, which usually leads to a bias in the solution of the convex relaxation. A common misconception is that if certain Restricted Isometry conditions are fulfilled, then both problem have the same solution, but this is only in the case when there is no noise, i.e.~if $b$ in \eqref{t5} is of the form $Ax_0$ where $x_0$ is sparse. For a deeper discussion of these drawbacks we refer to \cite{ourselves}. To remedy the problem with bias, there has recently been two independent attempts at finding convexifications closer to the original functional, namely \cite{larsson2016convex} for minimizing \eqref{f667} (in combination with additional restrictions) and \cite{soubies2015continuous} for minimizing \eqref{t5} as is. In this paper we find a unifying framework and significantly extend the existing theory.

Figure \ref{figintro} highlights these issues in one variable; let $|x|_0$ the function equalling 1 on $\R\setminus\{0\}$ and zero at $x=0$. In red we see the functional $|x|_0+\frac{1}{2}|x-1|^2$ (which is a particular case of both \eqref{t5} and \eqref{f667} in dimension 1), in blue its convex envelope and in pink the convex relaxation $|x|+\frac{1}{2}|x-1|^2$. Clearly the global minimum of the red and blue coincide, but the global minimum of the convex relaxation is different.

We now outline the main contributions of this paper in greater detail. Consider any functional of the form
\begin{equation}\label{t553} f(x)+\frac{1}{2}\|x-d\|^2_\V\end{equation}
where $\V$ is an arbitrary separable Hilbert space and $f$ any non-negative functional on $\V$. We introduce a transform $\S_\gamma$, where $\gamma>0$ is a parameter, which is designed so that $\S_\gamma^2(f)(x)+\frac{\gamma}{2}\|x\|^2$ is the l.s.c.~convex envelope of $f(x)+\frac{\gamma}{2}\|x\|^2$, and show that the l.s.c.~convex envelope of the functional in \eqref{t553} is
\begin{equation}\label{t55}\S_1^2(f)(x)+\frac{1}{2}\|x-d\|^2.\end{equation} Values $\gamma\neq 1$ will mainly be of interest in Part III, and we simply write $\S$ in place of $\S_1$. Note that the shape of the convex envelope is completely independent of $d$. The functionals $\S(f)$ and $\S^2(f)$ are closely related to the Moreau-envelope, Lasry-Lions approximants or proximal hulls, which we elaborate more on in Section \ref{son}. In Section \ref{ex} we provide numerous examples of $\S^2(f)$ for various functionals acting on matrices as well as vectors. We also provide a number of general results to simplify the computation of $\S^2(f)$.

Section \ref{finer} considers finer properties of l.s.c.~convex envelopes. The computation of the l.s.c.~convex envelope of $f(x)+\frac{1}{2}\|x\|^2$ can be thought of as stretching plastic foil from below onto the graph of $f(x)+\frac{1}{2}\|x\|^2$ (see Figure \ref{fig2intro}). Consider a point $x$ where the plastic foil is not in contact with the graph, i.e. where $\S^2(f)(x)<f(x)$. It is intuitively obvious that the plastic foil, i.e. the graph of $\S^2(f)(x)+\frac{1}{2}\|x\|^2$, has some direction in which it is affine linear, and thus $\S^2(f)$ should have some direction in which the curvature is $-1$. This is surprisingly difficult to show, and it is the main result of Section \ref{finer} that the statement is true, which reads as follows.

\begin{theorem}\label{o9}
Let $f$ be a weakly l.s.c.~$[0,\infty]$-valued functional on a separable Hilbert space $\V$, and pick $x_0\in\V$. We then either have that $f(x_0)=\S^2(f)(x_0)$, or there exists a unit vector $\nu$ and $t_0>0$ such that the function $h(t)= \S^2(f)(x_0+t\nu)$ has second derivative $-1$ on $(-t_0,t_0)$.
\end{theorem}

This implies that the l.s.c.~convex envelope of \eqref{t553} at each point either touches the original functional, or has a direction in which it is locally affine. Despite the wealth of results on l.s.c.~convex envelopes, this result seems to be new, although in the review process it has been brought to my attention that in the finite dimensional case, the statement is shown in the PhD-thesis \cite{lucet}. In either case, the proof given here is a simple extension of  a theorem due to Arne Br\o ndsted \cite{brondsted1966milman} in a short notice from 1966, which seems to have remained unnoticed by the community.

Semi-algebraicity of the $\S_\gamma-$transform is considered in Section \ref{semi}, since it was shown in \cite{attouch2013convergence} that this is sufficient for the forward backward splitting method to converge in the non-convex setting. This concludes the first part of the paper, titled ``general theory''.

\begin{figure}
\centering
\includegraphics[width=0.8\linewidth]{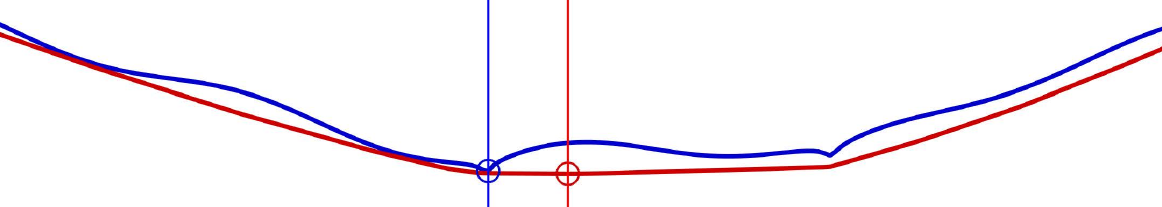}\\
\includegraphics[width=0.48\linewidth,height=5cm]{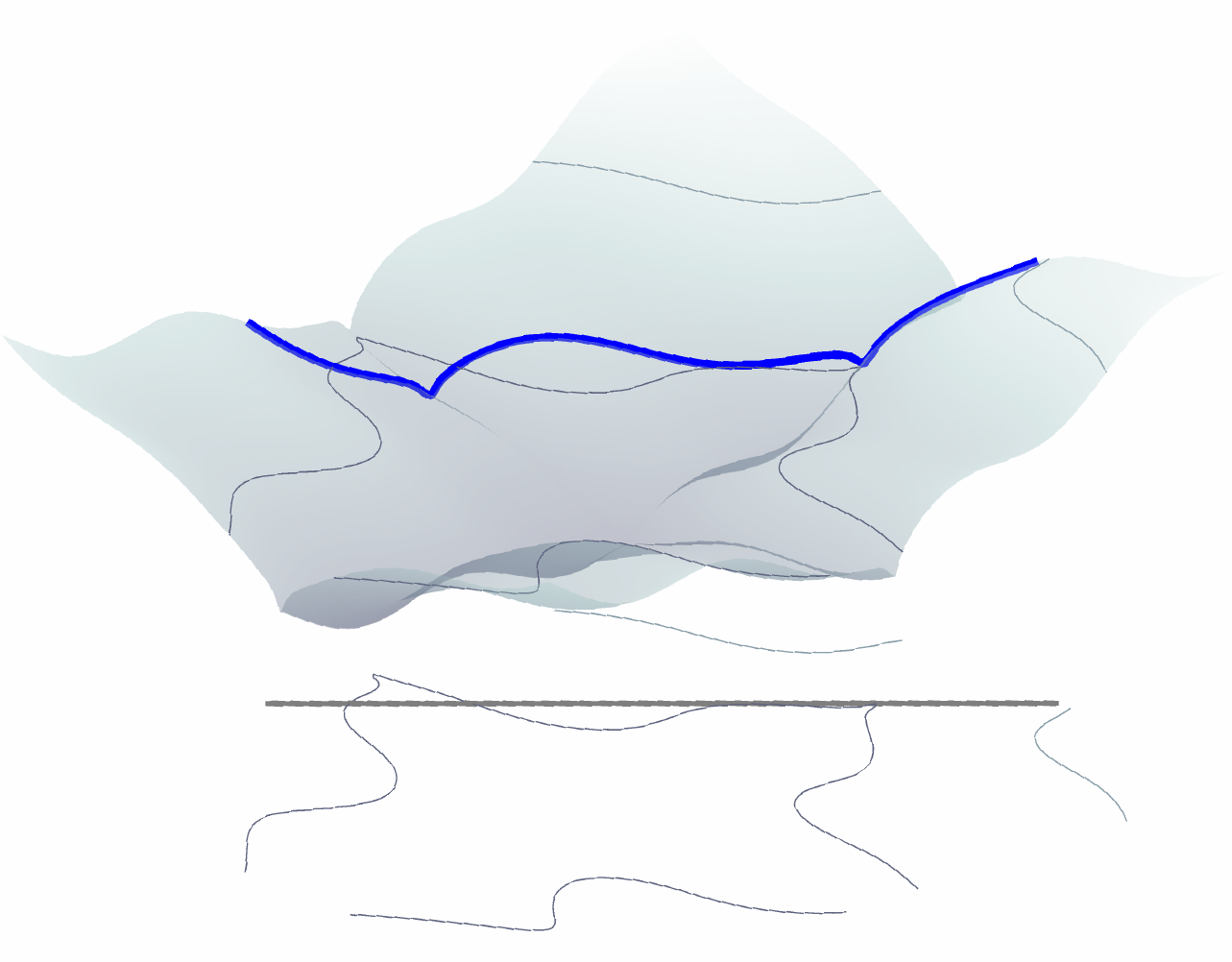}
\includegraphics[width=0.48\linewidth,height=5cm]{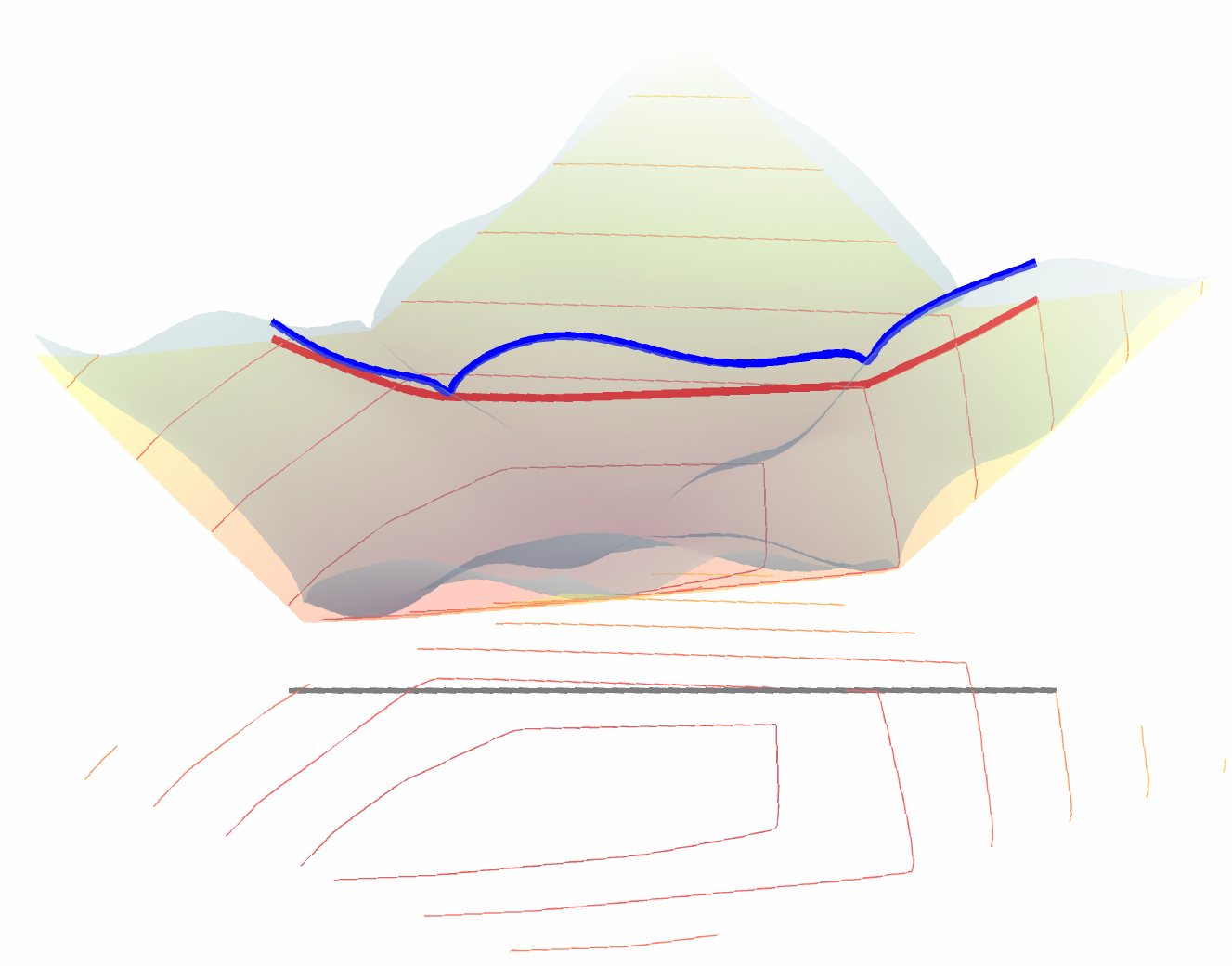}
\caption{Illustration of a non-convex optimization problem with linear constraints. The bottom left panel shows a non-convex functional along with its level sets. The gray line represents the subspace we are interested in, and the blue curve the values of the functional restricted to the subspace. The bottom right panel shows the same setup, but here the convex envelope is shown as well in orange/yellow. The values of the convex envelope over the subspace is shown in the red curve. The top figure shows a one-dimensional plot of the values of the original functional and the convex envelope evaluated on the subspace. The respective minima are shown by circles and highlighted by the vertical lines.}
\label{fig2intro}
\end{figure}

The remainder of the paper is divided in two parts corresponding to the prototype functionals \eqref{t5} and \eqref{f667}. These are rather different, and to explain why note that $\S(f)$ can be computed explicitly only if the global minimum of the original functional \eqref{t553} can be found explicitly, as in the case of \eqref{f667} but not \eqref{t5}. Therefore, the problem of minimizing \eqref{f667} only becomes difficult in combination with additional restrictions. Suppose e.g. that we want to minimize \eqref{t553} over some subspace $\M\subset\V$ or say that we wish to minimize $f(x)+\frac{1}{2}\|x-d\|^2+c(x)$ where $c$ is a convex functional related to any additional prior information, (see Section 4 in \cite{larsson2016convex} for concrete examples). In both cases we end up with minimization problems with no closed form solution. Replacing $f$ with $\S^2(f)$ then gives us a convex problem, similar to the original one, which can be addressed with standard convex approximation schemes like the projected subgradient method, dual ascent, ADMM or forward-backward splitting. It is often the case that the minimum of the ``convexified'' problem coincides with the minimum of the non-convex problem, which is easily verified by simply checking if $f(x)=\S^2(f)(x)$ holds at the point of convergence. It is important however to realize that this is not always the case, as Figure \ref{fig2intro} demonstrates. However, Figure \ref{fig3} in Section \ref{secap1} shows the same functional with a different subspace on which the two minima does coincide. We elaborate further on this in Section \ref{secap1}.  It is not the aim of the present paper to provide recommendations for which algorithm to use to solve a specific application, and the best candidate will certainly depend on the particular situation. Nevertheless, several of the algorithms mentioned above requires the ability to compute the so called proximal operator, and we provide theory for this in Section \ref{prox}, which concludes Part II of the paper, titled ``applications with additional priors''.

Part III of the paper is devoted to the problem of minimizing
\begin{equation}\label{t55312} f(x)+\frac{1}{2}\|Ax-d\|^2_{\V}\end{equation} where $A$ is any linear transformation. We assume that $\V$ is such that $\S^2_\gamma(f)$ is computable, but due to the linear transformation $A$, the functional
\begin{equation}\label{t55312c} \S_\gamma^2(f)(x)+\frac{1}{2}\|Ax-d\|^2,\quad \gamma=1,\end{equation} will not equal the convex envelope of \eqref{t55312}, which we assume is untractable, as in the case of \eqref{t5}. The parameter $\gamma$ now becomes a valuable tool as it tunes the curvature of $\S_\gamma^2(f)$. The expression \eqref{t55312c} is illustrated (in one dimension and for values of $|A|^2>\gamma$ (left) and $|A|^2<\gamma$ (right)) in Figure \ref{fig3intro}. The circles represent global minima of the respective functions.

Generalizing the left figure, we assume in Section \ref{secunderestimate} that $\gamma$ is below the square of the lowest singular value of $A$. We prove that the functional \eqref{t55312c} is a convex functional below \eqref{t55312}, and hence minimization of \eqref{t55312c} will produce a minimizer which, although not necessarily equal to the minimizer of the original problem, likely is closer than that obtained by other convex relaxation methods (if such at all are available). Moreover, the minimizer of the original and modified problem do coincide whenever $f(x)=\S_\gamma^2(f)(x)$, which often is easily checked in practice. An example of when this happens, similar to Figure \ref{fig3intro}, is shown in Figure \ref{fig1} in Section \ref{secunderestimate}.

\begin{figure}
\centering
\includegraphics[width=0.48\linewidth]{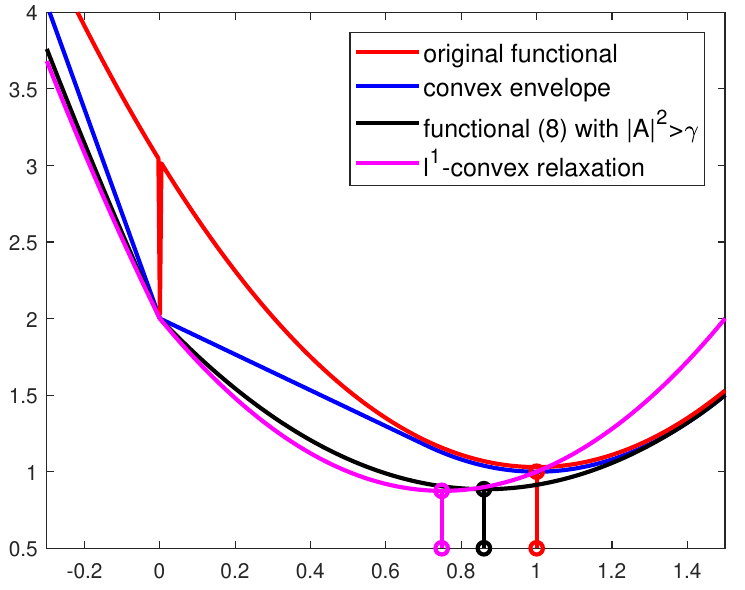}
\includegraphics[width=0.48\linewidth]{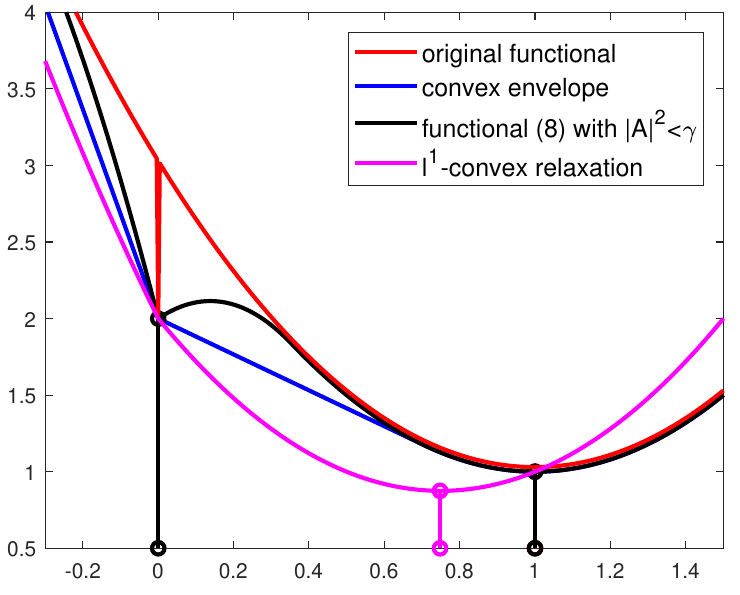}
\caption{The same setup as in Figure \ref{figintro}, but with an additional functional in black illustrating \eqref{t55312c} in case $A^*A\geq \gamma I$ (left) and $A^*A\leq \gamma I$ (right). See Section \ref{secex2} for a more detailed description.}
\label{fig3intro}
\end{figure}

For the problem \eqref{t5}, $A$ is usually a matrix with a large kernel, and the smallest singular value is 0, which rules out the above approach. In Section \ref{secoverestimate} we consider the case $\gamma>\|A\|^2$, generalizing the situation in the right picture of Figure \ref{fig3intro}. We can then show that \eqref{t55312c} is a continuous (but not everywhere convex) functional with the following desirable properties; $i)$ \eqref{t55312c} lies between \eqref{t55312} and its l.s.c. convex envelope, $ii)$ any local minimizer of \eqref{t55312c} is a local minimizer of \eqref{t55312}, $iii)$ the global minimizers of \eqref{t55312c} and \eqref{t55312} coincide (see Proposition \ref{pover} and Theorem \ref{tover1}). We remark that, despite not being convex, critical points of \eqref{t55312c} can be found using e.g. the forward-backward splitting method \cite{attouch2013convergence,bolte2014proximal}. The situation in Section \ref{secoverestimate} is thus drastically different from the previous scenarios; whether a global minimizer of the original problem is found depends only on the starting point for the algorithm seeking local minimizer. This latter part of the paper is inspired by \cite{soubies2015continuous}, which considers problem \eqref{t5}, and also contains a list of recent algorithms for finding local minima of functionals of the type considered above. A separate study of the methods developed here to this particular problem is also found in our recent contribution \cite{ourselves}.


\section*{Notation}
The set of $m\times n$ complex matrices, equipped with the Frobenius norm, is denoted $\m_{m,n}$. Throughout the paper, $\V$ and sometimes $\W$ denote separable Hilbert spaces (possibly finite dimensional). Let $\mathcal{B}_2(\V,\W)$ denote all Hilbert-Schmidt operators with the Hilbert-Schmidt norm. We remark that in case $\V=\C^n$ and $\W=\C^m$ with the canonical norms, then $\mathcal{B}_2(\V,\W)$ is readily identified with $\m_{m,n}$ with the Frobenius norm. The singular value decomposition (SVD) of a given $A\in\m_{m,n}$ is denoted $A=U\Sigma V^*$, where we choose  $V\in\m_{n,n}$, $\Sigma\in \m_{n,n}$ and $U\in\m_{m,n}$. The vector of singular values (i.e. the elements on the diagonal of $\Sigma$) is then denoted by $\sigma.$ Note that we thus define the singular values such that the amount of singular values equals the dimension of $\V$. More generally, given any operator $A$ acting on infinite dimensional spaces, we can pick singular vectors $(u_{j})_{j=1}^\infty$ and $(v_{j})_{j=1}^\infty$ such that \begin{equation}\label{SVD}A=\sum_{j=1}^\infty\sigma_j(A)u_{j}\otimes v_{j}\end{equation}
where $\sigma_j(A)$ are the singular values (ordered decreasingly) and $u_{j}\otimes v_{j}(x)=u_{j}\scal{x, v_{j}}$. Moreover $(u_{j})_{j=1}^\infty$ can be taken to be an orthonormal sequence in $\W$ and $(v_{j})_{j=1}^\infty$ to be an orthonormal basis in $\V$ (see e.g. Theorem 1.4 \cite{simon1979trace}). We follow the matrix theory custom of numbering the singular vectors starting at 1, as opposed to 0 which is more common in operator theory.

$\h (\V)$ will denote the subspace of $\B_2(\V,\V)$ of self-adjoint (Hermitian) operators, and $\lambda(X)$ the vector of eigenvalues of a given $X\in \h (\V)$. In case $\V$ has finite dimension $n$, so that $\B_2(\V,\V)$ is identified with $\m_{n,n}$, we simply write $\h_{n}$.

$\R^{d}$ for $d=\infty$ is identified with $\ell^2(\N)$. Given $x\in\R^d$, $\|x\|_0$ denotes the amount of non-zero elements (by abuse of notation since this is not a norm), and $\|x\|_2$ the canonical norm. We abbreviate lower semi-continuous by l.s.c., and we denote by $\mathsf{dom}(f)$ the set of points where the functional $f$ is finite. Both $CE(f)$ and $f^{**}$ will denote the l.s.c convex envelope of a functional $f$.

$\S_{\V,\gamma}$ is the $\S$-transform computed with the scalar product of $\V$ and parameter $\gamma$. Usually $\V$ is omitted from the notation and furthermore when $\gamma=1$ we simply write $\S$.

\section{Part I; general theory.}
\subsection{The $\S$-transform}\label{son}

Let $\V$ be a separable Hilbert space over $\R$ or $\C$, such as $\C^n$ with the canonical norm $\|x\|_2^2=\sum_{j=1}^n |x_j|^2$ or $\m_{m,n}$, equipped with the Frobenius norm which we denote $\|X\|_F$. All Hilbert spaces over $\C$ are also Hilbert spaces over $\R$ with the scalar product $ \scal{ x,y}_\R=\Re \scal{ x,y}$, and hence it is no restriction to assume that $\V$ is a real Hilbert space wherever needed. Even if $\V$ is a Hilbert space over $\C$, we will implicitly assume that the scalar product is $\scal{ x,y}_\R$.

Given any functional $g:\V\rightarrow \R\cup\{\infty\}$ the Legendre transform (or Fenchel conjugate) is defined as  \begin{equation}\label{Legendre}\L(g)(y)=g^*(y):=\sup_x  \scal{ x,y}-g(x).\end{equation}
We remind the reader that $g^*$ is l.s.c convex and that $g^{**}$ equals the l.s.c.~convex envelope of $g$, by the Fenchel-Moreau theorem (see e.g.~Proposition 13.11 and 13.39 in \cite{bauschke2011convex}).
Given a parameter $\gamma>0$, we now introduce the transform $\S_\gamma$ defined as follows: \begin{equation}\label{defS}\S_\gamma(f)(y):=\L \para{f(\cdot)+\frac{\gamma}{2}\|\cdot\|^2}(\gamma y)-\frac{\gamma}{2}\|y\|^2=\sup_x -f(x)-\frac{\gamma}{2}\fro{x-y}.\end{equation}
We denote $\S_\gamma\circ\S_\gamma$ by $\S_\gamma^2$. \textcolor{red}{The above formula had an error in previous versions. The article \cite{carlsson2018convex} is a condensed and improved version of this article, in which we denote $\S^2_\gamma$ by $\mathcal{Q}_\gamma$ and call it the quadratic envelope. We will not update notation/terminology in this article.} Note the direct formula \begin{equation}\label{lasry1}\S_{\gamma}^2(f)(x)=\sup_y\para{\inf_w f(w)+\frac{\gamma}{2}\fro{w-y}}-\frac{\gamma}{2}\fro{x-y}, \end{equation} so $\S_\gamma^2(f)$ can also be seen as an inf-convolution followed by a sup-convolution with $\pm\frac{\gamma}{2}\|\cdot\|^2$.
The parameter $\gamma$ basically tunes the maximum negative curvature of $\S_\gamma^2(f)$, which we will show in Section \ref{finer} (Theorem \ref{propconcave}). When $\gamma=1$ we simply write $\S$ as opposed to $\S_\gamma$.

It is clear from the second line of \eqref{defS} that $\S(f)$ is simply the negative of the famous Moreau-envelope. However, the double Moreau-envelope does not equal $\S^2(f)$, and is not connected with convex envelopes. We do have \begin{equation}\label{lasry}\begin{aligned}&\S_{1/s}\S_{1/t}(f)(x)=-\para{\inf_y-\para{\inf_w f(w)+\frac{1}{2t}\fro{w-y}}+\frac{1}{2s}\fro{x-y}}=\\& =\sup_y\para{\inf_w f(w)+\frac{1}{2t}\fro{w-y}}-\frac{1}{2s}\fro{x-y}, \end{aligned}\end{equation}
which, for parameters $s<t$, is called the Lasry-Lions approximation of $f$ \cite{lasry1986remark}, which has been studied in the context of regularization of non-convex functionals. For $s=t$ it is also called the proximal hull in \cite{rockafellar2009variational} (see Example 1.44), and it is also studied in Section 6 of \cite{stromberg1996regularization} (with the notation $C(1)f$), mainly with focus on differentiability-results. It is also closely connected to the more general ``proximal average'', see e.g.~\cite{bauschke2008proximal,hare2009proximal}. However, it seems that the connection with convex envelopes has not been systematically studied, which is the main aim of this publication. The next proposition contains some basic observations on the behavior of $\S_\gamma$, and Theorem \ref{t1} contains the connection with l.s.c.~convex envelopes of $f(x)+\frac{\gamma}{2}\|x-d\|^2.$

\begin{proposition}\label{propprop}
Let $f$ be a $[0,\infty]$-valued l.s.c.~functional on a separable Hilbert space $\V$ and $\gamma>0$. Then $\S_\gamma(f)$ takes values in $(-\infty,0]$ and is continuous, whereas $\S_\gamma^2(f)$ is lower semi-continuous, takes values in $[0,\infty]$ and is continuous in the interior of $\mathsf{dom}(\S_\gamma^2(f))$.
\end{proposition}
\begin{proof}
The statement of the interchanging signs follows easily by the last line of \eqref{defS}, which also shows that $\S_\gamma(f)$ avoids $-\infty$. By \eqref{defS} it also follows that $\S_\gamma(f)$ (and $\S_\gamma^2(f)$) is the difference of an l.s.c.~convex functional and a quadratic term. With this in mind the continuity statements follows by standard properties of l.s.c.~convex functionals (see e.g.~Corollary 8.30 \cite{bauschke2011convex}). \end{proof}

The following result is the key result of this section, connecting the $\S_\gamma$-transform with l.s.c.~convex envelopes.

\begin{theorem}\label{t1}
Let $f$ be a $[0,\infty]$-valued functional on a separable Hilbert space $\V$. Then
$$\L\para{f(x)+\frac{\gamma}{2}\|x-d\|^2}(y)=\S_\gamma(f)\para{\frac{y}{\gamma}+d}+\frac{\gamma}{2}\fro{\frac{y}{\gamma}+d}-\frac{\gamma}{2}\|d\|^2$$
and
$$\L\para{\S_\gamma(f)\para{\frac{y}{\gamma}+d}+\frac{\gamma}{2}\fro{\frac{y}{\gamma}+d}-\frac{\gamma}{2}\|d\|^2}(x)=\S_\gamma^2(f)(x)+\frac{\gamma}{2}\|x-d\|^2.$$
In particular, $\S_\gamma^2(f)(x)+\frac{\gamma}{2}\|x-d\|^2$ is the l.s.c.~convex envelope of $f(x)+\frac{\gamma}{2}\|x-d\|^2$ and $0\leq \S_\gamma^2(f)\leq f$.
\end{theorem}
\begin{proof}
We have
\begin{align*}
&\L\para{f(x)+\frac{\gamma}{2}\|x-d\|^2}(y)=\sup_x \scal{ x,y}-f(x)-\frac{\gamma}{2}\|x-d\|^2=\\&=\sup_x -f(x)-\frac{\gamma}{2}\fro{x-\para{\frac{y}{\gamma}+d}}+\frac{\gamma}{2}\fro{\frac{y}{\gamma}+d}-\frac{\gamma}{2}\|d\|^2
\end{align*}
from which the first identity follows. Similarly
\begin{align*}
&\L\para{\S_\gamma(f)\para{\frac{y}{\gamma}+d}+\frac{\gamma}{2}\fro{\frac{y}{\gamma}+d}-\frac{\gamma}{2}\|d\|^2}(x)\\&=\sup_y  \scal{ x,y}-\S_\gamma(f)\para{\frac{y}{\gamma}+d}-\frac{\gamma}{2}\fro{\frac{y}{\gamma}+d}+\frac{\gamma}{2}\|d\|^2=\\&
=\sup_y  -\S_\gamma(f)\para{\frac{y}{\gamma}+d}-\frac{\gamma}{2}\fro{\frac{y}{\gamma}+d-x}+\frac{\gamma}{2}\|x-d\|^2=\S_\gamma^2(f)(x)+\frac{\gamma}{2}\|x-d\|^2.
\end{align*}
The statement about the convex envelope follows by the Fenchel-Moreau theorem, which also gives $\S_\gamma^2(f)(x)+\frac{\gamma}{2}\|x-d\|^2\leq f(x)+\frac{\gamma}{2}\|x-d\|^2$. This implies the latter part of the inequality $0\leq \S_\gamma^2(f)\leq f$, whereas the former has already been noticed in Proposition \ref{propprop}.
\end{proof}

The above theorem can also be applied to expressions of the form \begin{equation}f(x)+\frac{1}{2}\fro{A x-d}_2,\quad x\in \C^n\end{equation} upon renormalizing $\V$ using $A$, but we postpone the theory for this case to Part III, in particular Proposition \ref{pl0}. Finer properties of the $\S_\gamma$-transform are discussed in Section \ref{finer}. In the coming section we make a long list of computable $\S_\gamma$-transforms as well as provide general tools to compute such.
We end this section with some observations about the behavior of $\S_\gamma^2(f)$ as a function of $\gamma$.

\begin{proposition}\label{pnew}
Let $f$ be a l.s.c.~$[0,\infty]$-valued functional. Then $\S_\gamma^2(f)(x)$ is increasing as a function of $\gamma$. Moreover, \begin{equation}\label{ggamma1}\lim_{\gamma\rightarrow \infty}\S_\gamma^2(f)(x)=f(x)\end{equation}
whereas the limit as ${\gamma\rightarrow 0^+}$ equals a convex minimizer of $f$ above the convex envelope of $f$.
\end{proposition}

We remark that $\lim_{\gamma\rightarrow 0^+}\S_\gamma^2(f)$ equals the l.s.c. convex envelope of $f$ for all the examples in Section \ref{ex}, but this is not necessarily the case in general, which is a surprise at least for the author. To see this, consider $P=\{x\in\R^2:~x_1>0,~x_2=\sqrt{x_1}\}$, $Q=\{x\in\R^2:~x_1>0,~0<x_2\leq\sqrt{x_1}\}\cup\{0\}$ and $f=\iota_{P}$, where $\iota_P$ is the indicator functional of $P$. It is easy to see that the l.s.c. convex envelope of $\iota_P$ equals $\iota_{cl(Q)}$ (where $cl$ denotes closure), whereas some thinking reveals that $\lim_{\gamma\rightarrow 0^+}\S_\gamma^2(f)=\iota_Q$. However, if e.g. $\V$ is finite dimensional and $\lim_{\gamma\rightarrow 0^+}\S_\gamma^2(f)$ is everywhere finite, then it is automatically continuous (Corollary 8.30 in \cite{bauschke2011convex}) and hence it must equal the l.s.c.~convex envelope of $f$.


\begin{proof}
If $\gamma_1>\gamma_2$ then $\S_{\gamma_2}^2(f)(x)+\frac{\gamma_1}{2}\|x\|^2$ equals the l.s.c.~convex functional $\S_{\gamma_2}^2(f)(x)+\frac{\gamma_2}{2}\|x\|^2$ plus the term $\frac{\gamma_1-\gamma_2}{2}\|x\|^2$, so it is l.s.c.~and convex. In view of $\S_{\gamma_2}^2(f)\leq f$ it also lies below $f+\frac{\gamma_1}{2}\|x\|^2$, and so we conclude that
$$\S_{\gamma_2}^2(f)(x)+\frac{\gamma_1}{2}\|x\|^2\leq CE(f+\frac{\gamma_1}{2}\|x\|^2)= \S_{\gamma_1}^2(f)(x)+\frac{\gamma_1}{2}\|x\|^2,$$ where $CE$ denotes the l.s.c.~convex envelope. The first claim follows.
For \eqref{ggamma1}, it suffices to show that $\lim_{\gamma\rightarrow \infty}\S_{\gamma}^2(f)(d)=f(d)$ for all $d\in\V$. To this end, let $\xi<f(d)$ be arbitrary. Since $f$ is l.s.c.~the set $\{x:f(x)>\xi\}$ is open and, as $f\geq 0$, it follows that we can pick $\gamma$ such that $$\xi-\frac{\gamma}{2}\|x-d\|^2\leq f(x),\quad x\in\V.$$ Thus the functional identically equal to $\xi$ is a l.s.c.~convex function below $f(x)+\frac{\gamma}{2}\|x-d\|^2$, and hence its l.s.c.~convex envelope is bigger than $\xi$. By Theorem \ref{t1} (evaluated at $x=d$), we conclude $\S_{\gamma}^2(f)(d)\geq \xi$, from which the desired result follows.


Concerning the limit as ${\gamma\rightarrow 0^+}$, set $g(x)=\lim_{\gamma\rightarrow 0^+}\S_\gamma^2(f)(x)$ which exist by the first part of this proposition. Since $$g(x)=\lim_{\gamma\rightarrow 0^+}\S_\gamma^2(f)(x)=\lim_{\gamma\rightarrow 0^+}\S_\gamma^2(f)(x)+\frac{\gamma}{2}\|x\|^2=\lim_{\gamma\rightarrow 0^+}CE(f+\frac{\gamma}{2}\|\cdot\|^2)(x)\geq CE(f)$$
we see that $g$ is the limit of a decreasing sequence of convex functions, hence it is also convex (Proposition 8.16 \cite{bauschke2011convex}) and clearly $g\leq f$ by Theorem \ref{t1}.
\end{proof}

\subsection{Examples of $\S$-transforms}\label{ex}

For practical purposes, Theorem \ref{t1} is only useful if $\S_\gamma^2$ has an explicit expression. This section contains a number of results that simplifies the computation of $\S_\gamma$-transforms, as well as numerous examples. The list of computable $\S$-transforms is by no means exhaustive and this section can be skipped by readers interested in other applications or theoretical aspects of the $\S$-transform, treated in later sections.

We begin by studying the functional $\|x\|_0$ on $\R$, which for clarity of notation we denote $|x|_0$, i.e. the function which is 0 at 0 and 1 elsewhere (see the red graph in Figure \ref{figbasic}). This seemingly innocent functional is relevant for both key problems \eqref{t5} and \eqref{f667}, which follows by noting that \begin{equation}\label{l0}\|x\|_0=\sum_{j=1}^n |x_j|_0\end{equation} and \begin{equation}\label{rank0}\rank(X)=\|\sigma(X)\|_0=\sum_{j=1}^n |\sigma_j(X)|_0.\end{equation}

\begin{figure}
\centering
\includegraphics[width=0.9\linewidth,height=5cm]{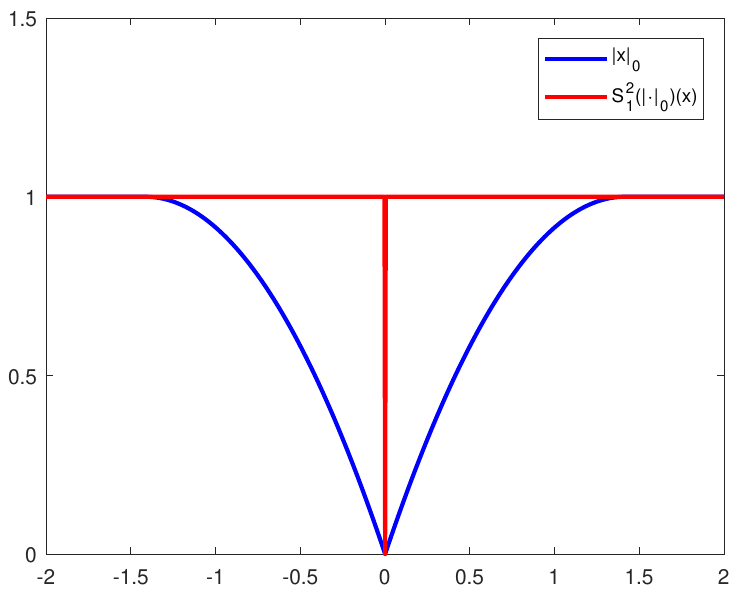}
\caption{Illustration of $|\cdot|_0$ (red) along with $\S_1^2(|\cdot|_0)$.}
\label{figbasic}
\end{figure}

\begin{ex}\label{ex1}
Let $\V=\R$ and consider $f(x)=\mu |x|_0$ where $\mu>0$ is a fixed parameter. Then
\begin{equation}\label{12}\S_\gamma(\mu |\cdot|_0)(y)=\sup_x -\mu |x|_0-\frac{\gamma}{2}(x-y)^2.\end{equation}
Clearly, the maximum is found either at $x=0$ or at $x=y$ which gives
\begin{equation}\label{14}\S_\gamma(\mu |\cdot|_0)(y)=\sup \{-\frac{\gamma y^2}{2},-\mu\}=-\min \{\frac{\gamma y^2}{2},\mu\}.\end{equation}
To compute $\S_\gamma^2(\mu |\cdot|_0)$, we repeat the process
\begin{align*}&\S_\gamma^2(\mu |\cdot|_0)(x)=\sup_{y}-(-\min \{\frac{\gamma y^2}{2},\mu\})- \frac{\gamma}{2} (x-y)^2=\sup_{y}\min \{\frac{\gamma y^2}{2},\mu\}- \frac{\gamma}{2} (x-y)^2
\end{align*}
Since $\min \{\frac{\gamma y^2}{2},\mu\}$ is constantly equal to its supremum value $\mu$ whenever $|y|\geq \sqrt{2\mu/\gamma}$, it follows that the maximum is attained at $y=x$ for all $x$ satisfying $|x|\geq \sqrt{2\mu/\gamma}$, which yields $\S_\gamma^2(\mu |\cdot|_0)(x)=\mu$. For the same reason the maximum is attained in $[-\sqrt{2\mu/\gamma},\sqrt{2\mu/\gamma}]$ whenever $|x|< \sqrt{2\mu/\gamma}$. Since the $y^2$-terms cancel in this segment, the functional to be maximized is linear there, and so the maximum must be obtained at $y=\pm\sqrt{2\mu/\gamma}$. It easily follows that
\begin{equation}\label{15}\S_\gamma^2(\mu |\cdot|_0)(x)=\mu-\frac{\gamma}{2}(|x|-\sqrt{2\mu/\gamma})^2\chi_{[-\sqrt{2\mu/\gamma},\sqrt{2\mu/\gamma}]}(x)=\mu-\para{\max\{\sqrt{\mu}-
\frac{\sqrt{\gamma} |x|}{\sqrt{2}},0\}}^2\end{equation}
where $\chi_S$ denotes the characteristic functional of a set $S$.
\end{ex}

The expression \eqref{15} has appeared e.g.~in \cite{larsson2016convex,soubies2015continuous}. The point here is that it allows us to compute the $\S$-transform of the more complicated cost functionals \eqref{l0} and \eqref{rank0}, when combined with the below propositions. We refer to Ch.~I.6 in \cite{conway2013course} for the basics of direct products of separable Hilbert spaces. We write $\S_\gamma=\S_{\V,\gamma}$ if there is a need to clarify which space is used to compute the transform.
\begin{proposition}\label{p1}
Let $(\V_j)_{j=1}^d$ where $d\in\N\cup\{\infty\}$ be separable Hilbert spaces and set $\V=\oplus_{j=1}^d \V_j$. Suppose that $f_j$ are $[0,\infty]$-valued functionals on $\V_j$ and set $F(x)=\sum_{j=1}^d f_j(x_j)$ where $x=\oplus_{j=1}^d x_j$ and $x_j\in\V_j$. Then $$\S_\gamma(F)(y)=\sum_{j=1}^d \S_{\V_j,\gamma}f_j(y_j).$$
\end{proposition}
\begin{proof} We have that
\begin{align*}&\S_\gamma(F)(y)=\sup_x -F(x)-\frac{\gamma}{2}\fro{x-y}=\sup_x\sum_{j=1}^d -f_j(x_j)-\frac{\gamma}{2}\fro{x_j-y_j}_{\V_j}=\sum_{j=1}^d \S_{\V_j,\gamma} f_j(y_j).\end{align*} If $d=\infty$ the interchange of sum and supremum is a bit delicate, but can be verified either by a short basic proof or using Rockafellar's interchange theorem \cite{rockafellar2009variational}.
\end{proof}

Combining this with Example \ref{ex1} we immediately get
\begin{equation}\label{l0000}\S^2(\|x\|_0)=\sum_{j=1}^d 1-\para{\max\{1-\frac{ |x|}{\sqrt{2}},0\}}^2, \quad x\in\R^d.\end{equation}
To derive a similar expression for \eqref{rank0}, we need von Neumann's trace inequality for operators on separable Hilbert spaces. We thus shift focus to functionals acting on the singular values of a matrix or, more generally, a Hilbert-Schmidt operator $X\in\B_2(\V_1,\V_2)$, (see e.g.~\cite{simon1979trace}). Set $d=\dim\V_1$ and note that the singular values of $X$ lies in the set $\R^{d}$ (see the Notation section), which we identify with $\ell^2(\N)$ in case $d=\infty$. The inequality then reads as follows:
\begin{theorem}\label{vonN}
Let $\V_1,~\V_2$ be any separable Hilbert spaces, let $X,Y\in \B_2(\V_1,\V_2)$ be arbitrary and denote their singular values by $\sigma_j(X)$, $\sigma_j(Y)$, respectively. Then $$\scal{X,Y}_{\B_2}\leq \sum_{j=1}^d \sigma_j(X)\sigma_j(Y)$$
with equality if and only if the singular vectors can be chosen identically.
\end{theorem}
The statement is well known for matrices but, surprisingly, the infinite dimensional version is nowhere to be found in the standard literature on operator theory, and we have also not been able to locate it in any scientific publication. For that reason, we include a proof in Appendix I. The next result shows how to ``lift'' expressions for $\S_\gamma^2$ from vectors to matrices.

\begin{proposition}\label{p2}
Let $\V_1,~\V_2$ be any separable Hilbert spaces. Suppose that $f$ is a permutation and sign invariant $[0,\infty]$-valued functional on $\R^{d}$, $d=\dim \V_1$, and that $F(X)=f(\sigma(X)),$ $X\in\B_2(\V_1,\V_2)$. Then $$\S_{\B_2,\gamma}(F)(Y)=\S_{\R^d,\gamma}(f)(\sigma(Y)).$$ In particular, this identity holds for all matrices.
\end{proposition}
\begin{proof}
Since
$\S_{\B_2,\gamma}(F)(Y)=\sup_X - f(\sigma(X))-\frac{\gamma}{2}\fro{X-Y}_{\B_2}$, von Neumann's inequality implies that the supremum is attained for an $X$ that shares singular vectors with $Y$. Hence
$$\S_{\B_2,\gamma}(F)(Y)=\sup_{\nu_1\geq \nu_2 \geq \ldots } - f(\nu)-\frac{\gamma}{2}\fro{\nu-\sigma(Y)}_2.$$
Due to the permutation and sign invariance of $f$, we can drop the restrictions on $\nu$ and so
$$\S_{\B_2,\gamma}(F)(Y)=\sup_{\nu} - f(\nu)-\frac{\gamma}{2}\fro{\nu-\sigma(Y)}_2=\S_{\R^d,\gamma}(f)(\sigma(Y)).$$
\end{proof}
It is now easy to determine the $\S$-transform of the rank-functional on matrices.
\begin{ex}\label{exrank}
By combining Proposition \ref{p1} and \ref{p2} with \eqref{rank0}, \eqref{14} and \eqref{15} we immediately get that
\begin{equation}\label{Srank}\S_\gamma(\mu\rank)(Y)=\sum_{j=1}^{d}\max \left\{\frac{-\gamma\sigma_j(Y)^2}{2},-\mu\right\}\end{equation} and \begin{equation}\label{SSrank}\S_\gamma^2(\mu\rank)(X)=\sum_{j=1}^{d}\mu-\para{\max\{\sqrt{\mu}-\frac{\sqrt{\gamma}\sigma_j(X)}{\sqrt{2}},0\}}^2.\end{equation}
\end{ex}
Expressions \eqref{Srank} and \eqref{SSrank} first appeared in \cite{larsson2016convex}, but we include them to illustrate the use of Propositions \ref{p1} and \ref{p2}.

\subsubsection{$\S_\gamma$-transforms in weighted matrix-spaces}

In many applications it is desirable to replace the Frobenius norm with a weighted norm. In this section we show how this can be done for a particular class of weights.  Given $W\in\m_{m,n}$ with (strictly) positive entries, we let $\m_{m,n}^W$ be the Hilbert space obtained by introducing the norm $$\|X\|_W^2=\sum_{i,j}w_{i,j}|x_{i,j}|^2,$$ where e.g.~$w_{i,j}$ are the entries of $W$. In case $W=\textbf{1}$, i.e. $W$ is equal to one componentwise, we will simply write $\m_{m,n}$ as earlier. Suppose now that we are interested in computing $\S_{\m_{m,n}^W,\gamma}(f)$, where $f$ is such that $\S_{\m_{m,n},\gamma}(f)$ has an explicit expression. In general, this will only be possible if $W$ is a direct tensor, i.e. of the form
\begin{equation}\label{simpler}w_{i,j}=u_iv_j\end{equation}
where $u$ and $v$ are sequences of length $m$ and $n$ respectively. The following examples and proposition show how to do this. A linear operator between two spaces that is bijective and isometric will be referred to as unitary.

\begin{ex}\label{exmatrixes}
Under the assumption \eqref{simpler}, note that $$X\mapsto I_{\sqrt{v}}XI_{\sqrt{u}}$$ is unitary between $\m^W_{m,n}$ and $\m_{m,n}$, where e.g.~$I_{\sqrt{u}}$ is a diagonal matrix with $\sqrt{u}=(\sqrt{u_j})_{j=1}^n$. Also note that $I_{\sqrt{v}}:\m_{m,1}^v\rightarrow \C^m$ and $I_{\sqrt{u}}:\C^n\rightarrow \m_{n,1}^{1/u}$ are unitary, where $1/u$ refers to componentwise division. The space $\m_{n,1}^{1/u}$ is of course the same as $\C^n$ as a vector space, but with a different norm. In fact, if $e_1,\ldots,e_n$ denotes the canonical basis in $\C^n$, we have that $u_j=\sqrt{u_j}e_j$ ($j=1,\ldots,n$) defines an orthonormal basis in $\m_{n,1}^{1/u}$. Each matrix $X=(x_{i,j})\in\m_{m,n}^W$ defines an operator $X\in \B_2(\m_{n,1}^{1/u},\m_{m,1}^{v})$ by the usual matrix multiplication, i.e. $(X y)_i=\sum_j x_{i,j}y_j$. It is easy to see that $$\|X\|_{\B_2(\m_{n,1}^{1/u},\m_{m,1}^{v})}^2=\sum_{j=1}^n\|Xu_j\|_{\m_{m,1}^v}^2=\sum_{j=1}^n\sum_{i=1}^m u_jv_i|x_{i,j}|^2=\|X\|_{\m_{m,n}^W}^2.$$
\end{ex}

\begin{proposition}\label{propositionweighted}
Let $\V_1,~\tilde\V_1~\V_2$ and $\tilde\V_2$ be separable Hilbert spaces, let $I_1:\V_1\rightarrow \tilde\V_1$ be unitary and let $I_2:\tilde\V_2\rightarrow \V_2$ be unitary. Then the induced map $\I:\B_2(\tilde\V_1,\tilde\V_2)\rightarrow \B_2(\V_1,\V_2)$ given by $\I(X)=I_2XI_1$ is unitary.

Moreover, let $f$ be an $[0,\infty]$-valued functional on $\B_2(\tilde\V_1,\tilde\V_2)$. Then $$\S_{\B_2(\tilde\V_1,\tilde\V_2),\gamma}(f)(Y)=\S_{\B_2(\V_1,\V_2),\gamma}(f\circ\I^{-1})(\I(Y))$$ and
$$(\S_{\B_2(\tilde\V_1,\tilde\V_2),\gamma})^2(f)(X)=(\S_{\B_2(\V_1,\V_2),\gamma})^2(f\circ\I^{-1})(\I(X))$$
\end{proposition}
\begin{proof}
The first statement is immediate by the definition of the Hilbert-Schmidt norm. The first identity follows from the calculation \begin{align*}&\S_{\B_2(\tilde\V_1,\tilde\V_2),\gamma}(f)(Y)=\sup_{X\in\B_2(\tilde\V_1,\tilde\V_2)} -f(X) -\frac{\gamma}{2}\|X-Y\|^2_{\B_2(\tilde\V_1,\tilde\V_2)}=\\&=\sup_{X\in\B_2(\tilde\V_1,\tilde\V_2)} -f(\I^{-1}(\I X)) -\frac{\gamma}{2}\|\I(X-Y)\|^2_{\B_2(\V_1,\V_2)}=
\\&=\sup_{Z\in\B_2(\V_1,\V_2)} -f(\I^{-1}(Z)) -\frac{\gamma}{2}\|Z-\I Y\|^2_{\B_2(\V_1,\V_2)}=\S_{\B_2(\V_1,\V_2),\gamma}(f\circ\I^{-1})(\I Y),\end{align*}
and the latter is a consequence of applying the former twice.
\end{proof}

\begin{ex}\label{exmatrixcont}
We continue Example \ref{exmatrixes}. Set $\V_1=\C^n$, $\tilde\V_1=\m_{n,1}^{1/u}$, $I_1=I_{\sqrt{u}}$, $\V_2=\C^m$, $\tilde\V_2=\m_{m,1}^v$, $I_2=I_{\sqrt{v}}$ and $f=\rank$.
Note that $f\circ\I^{-1}=f$ since left or right multiplication with invertible diagonal matrices do not change the rank. By Proposition \ref{propositionweighted} and Example \ref{exrank} we conclude that
\begin{equation}\label{SrankGen}\S_{\m_{m,n}^W}(\rank)(Y)=\S_{\m_{m,n}}(\rank)(\I(Y))=\sum_{j=1}^{n}\max \left\{\frac{-\sigma_j(I_{\sqrt{v}} Y I_{\sqrt{u}})^2}{2},-1\right\}\end{equation} and \begin{equation}\label{SSrankGen}\S_{\m_{m,n}^W}^2(\rank)(X)=\S_{\m_{m,n}}^2(\rank)(\I(X))=\sum_{j=1}^{n}1-\para{\max\{1-\frac{\sigma_j(I_{\sqrt{v}} X I_{\sqrt{u}})}{\sqrt{2}},0\}}^2,\end{equation}
generalizing \eqref{Srank} and \eqref{SSrank}.
\end{ex}

The expression \eqref{SSrankGen} is new to this paper and can be used e.g.~for applications in frequency estimation (see e.g.~\cite{andersson2017fixed}), which will be further investigated in depth elsewhere. However, the following example (as well as Example \ref{exweights1} and \ref{tuesday} explains the main idea).

\begin{ex}\label{exweights} Fix $n\in \N$ and let $H_f\in\m_{n,n}$ be the Hankel matrix generated by the sequence $f=(f_1,\ldots,f_{2n-1})\in\C^{2n-1}$. If one is interested in minimizing the rank of a Hankel matrix $H_f$ while at the same time not deviating far from some measurement $d\in\C^{2n-1}$, as is frequent in frequency estimation \cite{andersson2016fixed}, one option is to minimize the functional $\rank(X)+\frac{1}{2}\fro{X-H_d}_F$ over the set of Hankel matrices, (we consider minimization over subspaces in more detail in Part II, Example \ref{tuesday}). Setting $X=H_f$, the quadratic term $\fro{X-H_d}_F$ corresponds to a weighted misfit term of the form \begin{equation}\label{triangleweight}\fro{H_f-H_d}_F=\sum_{j=1}^{2n-1}(n-|j-n|)|f_j-d_j|^2,\end{equation} (see Figure \ref{figweights}, left) which is not the most natural quantity to minimize, as has been observed by many authors (e.g. \cite{gillard2013optimization}).
\begin{figure}
\begin{center}
\includegraphics[width=0.49\linewidth]{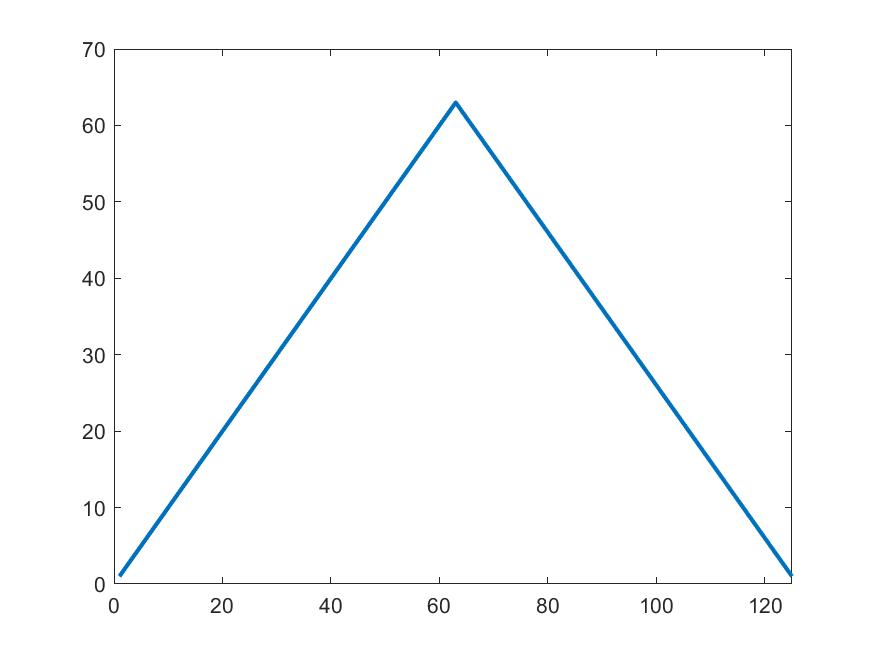}
\includegraphics[width=0.49\linewidth]{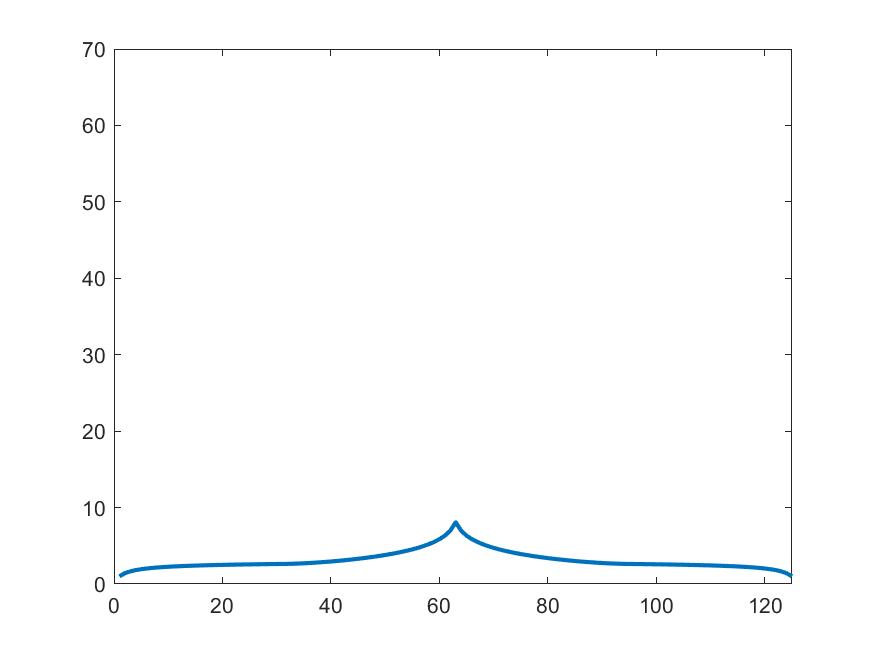}
\end{center}
\caption{Left; the weight appearing in \eqref{triangleweight} for $n=63$. Right; corresponding weight for \eqref{otherweight}.}\label{figweights}
\end{figure}
\end{ex}

\begin{ex}\label{exweights1} Continuing example \ref{exweights} we consider minimization of the functional $$\rank(X)+\frac{1}{2}\fro{X-H_d}_{W}$$ over the set of Hankel matrices, where we assume that $n=2k-1$ is odd and that  $w_{i,j}=u_iu_j$ with $u_i=\frac{1}{\sqrt{k-|i-k|}}$. By the above theory the l.s.c.~convex envelope is given by $$\sum_{j=1}^{n}1-\para{\max\{1-\frac{\sigma_j(I_{\sqrt{u}} X I_{\sqrt{u}})}{\sqrt{2}},0\}}^2+\frac{1}{2}\fro{X-H_d}_{W}.$$ Inserting $X=H_f$ in the quadratic term gives
\begin{equation}\label{otherweight}\fro{H_f-H_d}_W=\sum_{j=1}^{2n-1}\omega_j|f_j-d_j|^2,\end{equation} where $\omega_j$ is depicted in Figure \ref{figweights}, right. Compared with \eqref{triangleweight}, this weight is clearly much closer to a uniform flat weight (both weights \eqref{triangleweight} and \eqref{otherweight} start and end with the weight 1, so the scaling in Figure \ref{figweights} is fair). What the optimal choice of $u$ would be in order to yield as flat a weight as possible, is to our knowledge an open question.
\end{ex}

\subsubsection{Known ``model-order''}

Functionals of the type \eqref{l0} and \eqref{rank0} arise naturally if one looks for a ``sparse'' solution, but the degree of sparsity is not known, i.e. the number $M$ of non-zero parameters or the ``model-order''. In many applications, e.g.~rigid structure from motion, photometric stereo, optical flow \cite{larsson2016convex}, the model order is known and better results are obtained if this information is built into the functional to be minimized. This leads to consideration of functionals like

\begin{ex}\label{exfixedK}
In $\C^d$ define $\iota_M(x)=\left\{\begin{array}{cc}
                     0 & \|x\|_0 \leq M,\\
                     \infty & \text{ else.}
                   \end{array}
\right.$ and define $\tilde x$ to be a vector $x$ resorted so that $(|\tilde x_j|)_{j=1}^d$ is a non-increasing sequence. Then
\begin{align*}&\S(\iota_M)(y)= \sum_{j=M+1}^d -\frac{1}{2}|\tilde y_j|^2.\end{align*}
To see this, note that
$\S(\iota_M)(y)=\sup_x -\iota_M(x)-\frac{1}{2}\sum_{j=1}^d(x_j-y_j)^2,$ and it is clear that the optimal value of $x_j$ is $y_j$ if $|y_j|$ is among the $M$ greatest, and zero else.

The computation of $\S^2(\iota_M)$ is more involved. The expression is \begin{align*}&\S^2(\iota_M)(x)= \frac{1}{2k_*}\para{\sum_{j>M-k_*}|\tilde{x}_j|}^2-\frac{1}{2}\sum_{j>M-k_*}|\tilde x_j|^2\end{align*} where $k_*$ is a particular number between 1 and $M$.
\end{ex}
This is derived in \cite{andersson2016convex}, albeit without using the $\S$-transform explicitly and in the setting of matrices with fixed rank (see Example \ref{ale}). Nevertheless, the computations are easily adapted to $\iota_M$ as above. We now lift the above functional to the matrix case.

\begin{ex}\label{ale}
Let $\mathcal{M}_M\subset\m_{n,n}$ be the manifold of matrices of rank $\leq M$, and let $\iota_{\mathcal{M}_M}$ be the indicator functional of $\mathcal{M}_M$, i.e. the functional which is 0 on $\mathcal{M}_M$ and $\infty$ elsewhere. Letting $\iota_M$ be as above, note that $\iota_{\mathcal{M}_M}(X)=\iota_M(\sigma(X))$. Hence we can use Proposition \ref{p2} to see that $\iota_{\mathcal{M}_M}(X)$ has $\S$-transform $\S(\iota_M)(\sigma(Y))$ and \begin{equation}\label{lund}\S^2(\iota_{\mathcal{M}_M})(X)=\S^2(\iota_M)(\sigma(X)),\end{equation} we refer to \cite{andersson2016convex} or \cite{larsson2016convex} for more information on this particular functional. The latter reference investigates the present example and Example \ref{exrank} in a more general framework, looking at functionals of the form $g(\rank(X))$ where $g$ is a ``convex'' non-decreasing functional on the natural numbers (see eq. (5) in \cite{larsson2016convex} for a precise definition). They derive a feasible algorithm for computing $\S^2(g(\rank(X)))$, which in their nomenclature is denoted $\mathcal{R}_g(X)$ (see eq. (19)). It is easy to see that the same method can be adapted to also deal with functionals on $\R^n$ of the form $x\mapsto g(|\tilde x|)$, where $\tilde{x}$ is as in Example \ref{exfixedK}.

We remark that \eqref{lund} is valid with respect to unweighted $\m_{n,n}$. If we consider a weight $W$ as in Example \ref{exmatrixes}, then, in analogy with Example \ref{exmatrixcont}, we have
\begin{equation}\label{SSfixedrankGen}\S_{\m_{m,n}^W}^2(\iota_{\mathcal{R}_M})(X)=\S^2(\iota_M)(\sigma(I_{\sqrt{v}} X I_{\sqrt{u}})).\end{equation}

\end{ex}

\subsubsection{Positivity constraints}

We now look at functionals on eigenvalues rather than singular values. Let $\V$ be any separable Hilbert space, denote by $\h (\V) \subset \B_2(\V,\V)$ the space of self-adjoint (Hermitian) operators, and let $\lambda(X)$ denote the eigenvalues of a given $X\in\h (\V)$. In case $\V$ is of finite dimension $n$, so that $\B_2(\V,\V)\simeq \m_{n,n}$, we simply write $\h_{n}$. Keeping in mind that the singular values of a self-adjoint matrix are simply the modulus of the corresponding eigenvalues, the proof of Proposition \ref{p2} can easily be modified to give


\begin{proposition}\label{p33}
Let $\V$ be a separable Hilbert space. Suppose that $f$ is a permutation invariant functional on $\R^{\dim\V}$ and that $F:\h(\V)\rightarrow\R$ is given by $F(X)=f(\lambda(X))$. Then $$\S_{\h(\V),\gamma}(F)(Y)=\S_{\R^{\dim\V},\gamma}(f)(\lambda(Y)).$$
\end{proposition}
Suppose we are interested in positive matrices with low rank. In analogy with the previous developments, this calls for an investigation of the following functional.

\begin{ex}\label{ex11}
Set $\V=\R$, recall that $\chi_S$ is the characteristic functional of a given set $S$, whereas $\iota_S$ denotes the indicator functional, and set $$f(x)=\mu\chi_{(0,\infty)}(x)+\iota_{(-\infty,0)}$$. By a variation of the calculations in Example \ref{ex1}, we have
\begin{equation*}\S_\gamma(f)(y)= -\para{\min\left\{\frac{\sqrt{\gamma} y}{\sqrt{2}},\sqrt{\mu}\right\}}^2
\end{equation*}
and
\begin{equation*}\S_\gamma^2(f)(x)=\mu-\para{\max\{\sqrt{\mu}-\frac{\sqrt{\gamma} x}{\sqrt{2}},0\}}^2+\iota_{(-\infty,0)}(x)\end{equation*}
\end{ex}

\begin{ex}
Let $f$ be as above with $\mu=1$ and let $\mathbb{P}_{n}\subset\h_{n}$ be the set of positive matrices. Define $F$ on $\h_{n}$ by $$F(X)=\sum_{j=1}^n f(\lambda_j(X))=\rank(X)+\iota_{\{X\not\in \mathbb{P}_{n}\}}(X).$$ Then $$\S^2(F)(X)=\sum_{j=1}^n \S^2(f)(\lambda_j(X))=\sum_{j=1}^n 1-\para{\max\left\{1-\frac{ \lambda_j(X)}{\sqrt{2}},0\right\}}^2+\iota_{\{X\not\in \mathbb{P}_{n}\}}(X),$$ as follows by combining Proposition \ref{p1}, \ref{p33} with Example \ref{ex11}. This expression has been published previously in \cite{andersson2016fixed}, investigating applications to half-life parameter estimation.
\end{ex}

Finally, suppose we want to have at most $M$ positive eigenvalues and no negative ones.

\begin{ex}\label{exfixedKpos}
On $\R^d$ define $\iota_M^+(x)=\left\{\begin{array}{cc}
                     0 & \|x\|_0 \leq M \text{ and } x\geq 0,\\
                     \infty & \text{ else.}
                   \end{array}
\right.$. By a refinement of Example \ref{exfixedK} we have
\begin{align*}&\S(\iota_M^+)(y)= \frac{1}{2}\para{\sum_{j=1}^M |\max(\tilde y_j,0)|^2-\fro{y}}\end{align*}
where $\tilde{y}$ now denotes the vector obtained by reordering $y$ to a non-increasing vector.
\end{ex}
We omit a computation of $\S^2(\iota_M^+)$, because it is not needed for the evaluation of the proximal operator (see Proposition \ref{propprox}). The details are similar to those in \cite{andersson2016convex}.

To summarize this section, we have shown that the $\S$-transform is a useful tool for computing l.s.c.~convex envelopes of \eqref{t553} and simplified the computation of such convex envelopes in a number of known instances. We have also provided a number of new l.s.c.~convex envelopes of rather intricate functionals.
\subsection{Weak lower semi-continuity and finer properties of l.s.c. convex envelopes}\label{finer}

In this final section we prove Theorem \ref{o9}. We also give a result, based on an extension of the Milman theorem by Arne Br\o ndsted \cite{brondsted1966milman}, about the structure of l.s.c.~convex envelopes which seems relatively unknown. For this we need the concept of weak lower-semicontinuity, which is nothing but semi-continuity with respect to the weak topology of the underlying separable Hilbert space $\V$. We remind the reader that for convex proper functionals there is no difference (Theorem 9.1 \cite{bauschke2011convex}) between weakly l.s.c.~functionals and standard l.s.c.~functionals. Also, if $\V$ is finite dimensional and the topology is Hausdorff, the two topologies are the same (Exc. 18, Ch. IV.1 \cite{conway2013course}), so there is no difference in this case either. However, we wish to underline that the difficulty in proving the main result is present also in the finite-dimensional setting.

Examples of weakly l.s.c.~functionals include the support-cardinality functional $\|x\|_0=\#\{k\in\N:~x_k\neq 0\}$ in $\ell^2(\N)$, as well as the rank functional on $\mathcal B_2(\V_1,\V_2)$. In particular, if $\V_1=\C^n$ and $\V_2=\C^m$ with the canonical norms, then $\mathcal{B}_2(\V_1,\V_2)$ equals $\m_{m,n}$ with the Frobenius norm. For completeness, we include a proof of these claims in Appendix II. The main result of this section is the following theorem, whose proof comes at the end.

\begin{theorem}\label{propconcave}
Let $f$ be a weakly l.s.c.~$[0,\infty]$-valued functional on a separable Hilbert space $\V$. For each $x_0\in\V$ with $f(x_0)>\S_\gamma^2(f)(x_0)$ there exists a unit vector $\nu$ and $t_0>0$ such that the function $h(t)= \S_\gamma^2(f)(x_0+t\nu)$ has second derivative $-\gamma$ on $(-t_0,t_0)$.
\end{theorem}
The proof relies on a neat fact concerning weakly l.s.c.~convex envelopes which does not seem to have made its way into the modern literature on the subject. As mentioned earlier, it is a reformulation of Arne Br\o ndsted's extension of Milman's theorem. To state it, we remind the reader that a functional $g$ is coercive if and only if its (lower) level sets are bounded, (see e.g.~Proposition 11.11 \cite{bauschke2011convex}). Note that l.s.c.~convex envelopes of the type $\S^2(f)(x)+\frac{1}{2}\|x-d\|^2$ (for positive $f$) always are coercive, by virtue of Proposition \ref{propprop} and the quadratic term. Recall that $g^{**}$ is the l.s.c.~convex envelope of a given functional $g$ by the Fenchel-Moreau theorem. A function $f$ on $\R$ is called affine if it is of the form $f(t)=at+b$ with $a,b\in\R$.
\begin{theorem}\label{brondsted}
Let $g$ be a weakly l.s.c.~functional on a separable Hilbert space $\V$ such that $g^{**}$ is coercive. Given any $x\in\V$, we either have $g(x)=g^{**}(x)$ or there exists a unit vector $\nu$ and $t_0>0$ such that the function $h(t)= g^{**}(x_0+t\nu)$ is affine on $(-t_0,t_0)$.
\end{theorem}

We remark that both statements may hold simultaneously. The theorem should be considered in the light of that we may have $g^{**}(x)<g(x)$ and yet that the subdifferential of $g^{**}$ is empty. To prove Theorem \ref{brondsted} we recall some concepts from \cite{brondsted1966milman}. Given a convex function $f$ a point $x$ is called extremal if and only if $(x,f(x))$ is extremal for the epigraph of $f$, denoted $[f]$. Equivalently, $x$ is extremal if and only if $x\in\mathsf{dom}~f$ and $f$ is not affine on any relatively open segment containing $x$. Moreover $f_{ext}$ denotes the functional which equals $f(x)$ for all extremal points $x$ and $\infty$ else. As a consequence of Theorem 1 in \cite{brondsted1966milman} we have

\begin{theorem}\label{arne}
Let $g$ be a weakly l.s.c.~functional on a separable Hilbert space $\V$ such that $g^{**}$ is coercive, then $$[g^{**}_{ext}]\subset [g].$$
\end{theorem}
\begin{proof}
In the setting of \cite{brondsted1966milman} we let $E$ be the separable Hilbert space $\V$ with the weak topology. Since convex functionals are l.s.c.~with respect to the weak topology if and only if they are with respect to the norm topology, (see Theorem 9.1 \cite{bauschke2011convex}), it follows that the l.s.c~convex envelope of $g$ equals the weakly l.s.c.~convex envelope. In the notation of Theorem 1 of \cite{brondsted1966milman}, we can then take $f=g^{**}$ and the theorem states that $[f_{ext}]\subset [g_{cl}]$ where $g_{cl}$ is the greatest l.s.c.~minorant of $g$. Since $g$ is assumed to be l.s.c.~we have $g=g_{cl}$ and the desired inclusion follows. It remains to check that the conditions of Theorem 1 are fulfilled, which is that ``$g$ is inf-compact in some direction'' (with respect to the weak topology, referring to the terminology of \cite{brondsted1966milman}). For this it suffices to check that $g^{**}$ is inf-compact, i.e. that all level sets are bounded. The level sets of $g^{**}$ are closed and convex and since $g^{**}$ is assumed coercive they are also bounded. It follows that such level sets are compact in the weak topology, and the proof is complete.

\end{proof}

Based on this, we can now easily prove Theorem \ref{brondsted}.
\begin{proof}[Proof of Theorem \ref{brondsted}]
Since $g\geq g^{**}$, Theorem \ref{arne} clearly implies that $g(x)=g^{**}(x)$ for all extremal points $x$ for $g^{**}$. Consequently, if $g(x)= g^{**}(x)$ does not hold, then $x$ is not extremal for $g^{**}$ and the existence of $\nu$ follows by the definition of an extremal point for $g^{**}$.
\end{proof}

Next, we discuss what the theorem implies about minimizers of $g$ versus $g^{**}$. Denote by $G$ the set of global minimizers of $g$ and by $G^{**}$ the set of global minimizers of $g^{**}$.
\begin{corollary}\label{corminimizers}
Let $g$ be a weakly l.s.c.~functional on a separable Hilbert space $\V$ such that $g^{**}$ is coercive. Then $G^{**}$ is a closed bounded convex set containing $G$. Letting $G^{**}_{ext}$ denote the extremal points of $G^{**}$, we also have that $G^{**}_{ext}\subset G$. Finally, the closed convex hull of $G^{**}_{ext}$ equals $G^{**}$.
\end{corollary}
\begin{proof}
The convexity of $G^{**}$ and the inclusion $G\subset G^{**}$ are immediate. The boundedness of $G^{**}$ follows since $g^{**}$ is coercive. Let $x$ be in the closure of $G^{**}$, and let $c$ be the value of the global minimum. Then $g^{**}(x)\leq c$ follows by l.s.c.~, and the reverse inequality is obvious from the fact that $c$ is a global minimum. It follows that $x\in G^{**}$ and hence $G^{**}$ is closed.

The existence of points in $G^{**}_{ext}$ and the statement concerning the closed convex hull are now immediate consequences of the Krein-Milman theorem (see e.g.~\cite{conway2013course}) and the fact that bounded closed convex sets are weakly compact in separable Hilbert spaces (Theorem 3.33, \cite{bauschke2011convex}). It remains to prove that $G^{**}_{ext}\subset G$. Let $x_0\in G^{**}_{ext}$ suppose $x_0\not \in G$. Then Theorem \ref{brondsted} implies the existence of a direction $\nu$ on which $g^{**}$ is constant near $x_0$, contradicting that $x_0$ is an extremal point.
\end{proof}

\begin{proof}[Proof of Theorem \ref{propconcave}]
Set $g(x)=f(x)+\frac{\gamma}{2}\|x\|^2$. By Theorem \ref{t1} we have $\S_\gamma^2(f)(x)+\frac{\gamma}{2}\|x\|^2=g^{**}(x),$ by which it is immediate that $g^{**}$ is coercive, (since $\S_\gamma^2(f)\geq 0$ by Proposition \ref{propprop}). It also follows that $g(x_0)>g^{**}(x_0)$ and hence Theorem \ref{brondsted} implies that a unit vector $\nu$ exists such that $t\mapsto \S_\gamma^2(f)(x+t\nu)$ equals an affine function minus $\frac{\gamma}{2}t^2$ in a neighborhood of $t=0$. The desired statement follows.
\end{proof}

\subsection{The $\S$-transform and semi-algebraicity}\label{semi}
We briefly treat semi-algebraicity of $\S^2_\gamma(f)$, since it was shown in \cite{attouch2013convergence} that this is sufficient for the forward backward splitting method to converge in the non-convex setting. We remind the reader that a function on a finite dimensional space is semi-algebraic if its graph is a semi-algebraic set \cite{bochnak2013real}, although we follow the convention in \cite{attouch2013convergence} of including in the definition functions that can take the value $\infty$. In this case the graph is defined to be $\{(x,f(x)): x\in\mathsf{dom}(f)\}$, so these functions are also semi-algebraic in the sense of \cite{bochnak2013real}.
\begin{theorem}\label{tf}
If $\V$ is finite dimensional and $f$ is semi-algebraic, then so is $\S_\gamma(f)$.
\end{theorem}
\begin{proof}
We assume for simplicity that $\gamma=1$. It is a consequence of the Tarski-Seidenberg theorem that the set of semi-algebraic functions is closed under addition (see e.g.~Prop.~2.2.6 in \cite{bochnak2013real}), and similarly one can prove that the epigraph of a semi-algebraic function is a semi-algebraic set. If $f$ is semi-algebraic on $\R^n$, it follows that $g(x,y)=\scal{x,y}-(f(x)+\frac{1}{2}\fro{x})$ is semi-algebraic on $\R^{2n}$, and by the argument following Theorem 2.2 in \cite{attouch2013convergence} it follows that the Legendre transform of $f+\frac{1}{2}\fro{x}$ is semi-algebraic. The desired result now follows since this function minus $\frac{\gamma}{2}\fro{y}$ equals $\S_\gamma(f)(y)$ by \eqref{defS}.
\end{proof}
Finally we remark that all functionals (that operate on finite-dimensional spaces) in Section \ref{ex} are semi-algebraic, and in particular this then holds for $\S^2_\gamma(\rank)$. To see this, note that $x\mapsto \|x\|_0$ is semi-algebraic, and that $\rank(X)=\|\sigma(X)\|_0$ where $\sigma(X)$ is the vector of singular values. Moreover, the singular values of a matrix is a vector-valued semi-algebraic function. Since semi-algebraic functions are closed under composition, it follows that the rank functional is semi-algebraic.

\section{Part II; applications with additional priors}
\subsection{Minimization over convex subsets}\label{secap1}
\begin{figure}
\centering
\includegraphics[width=0.8\linewidth]{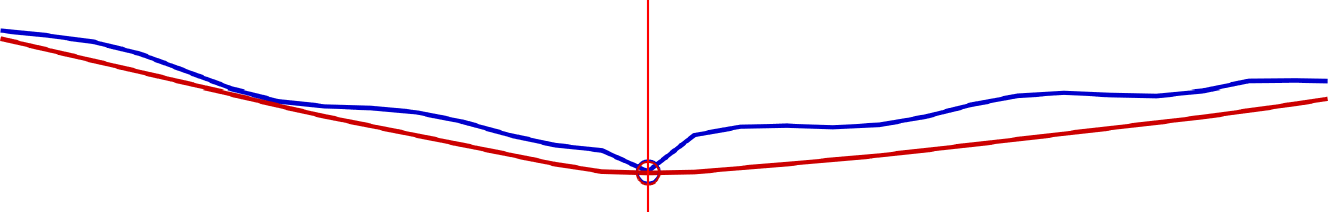}\\
\includegraphics[width=0.48\linewidth]{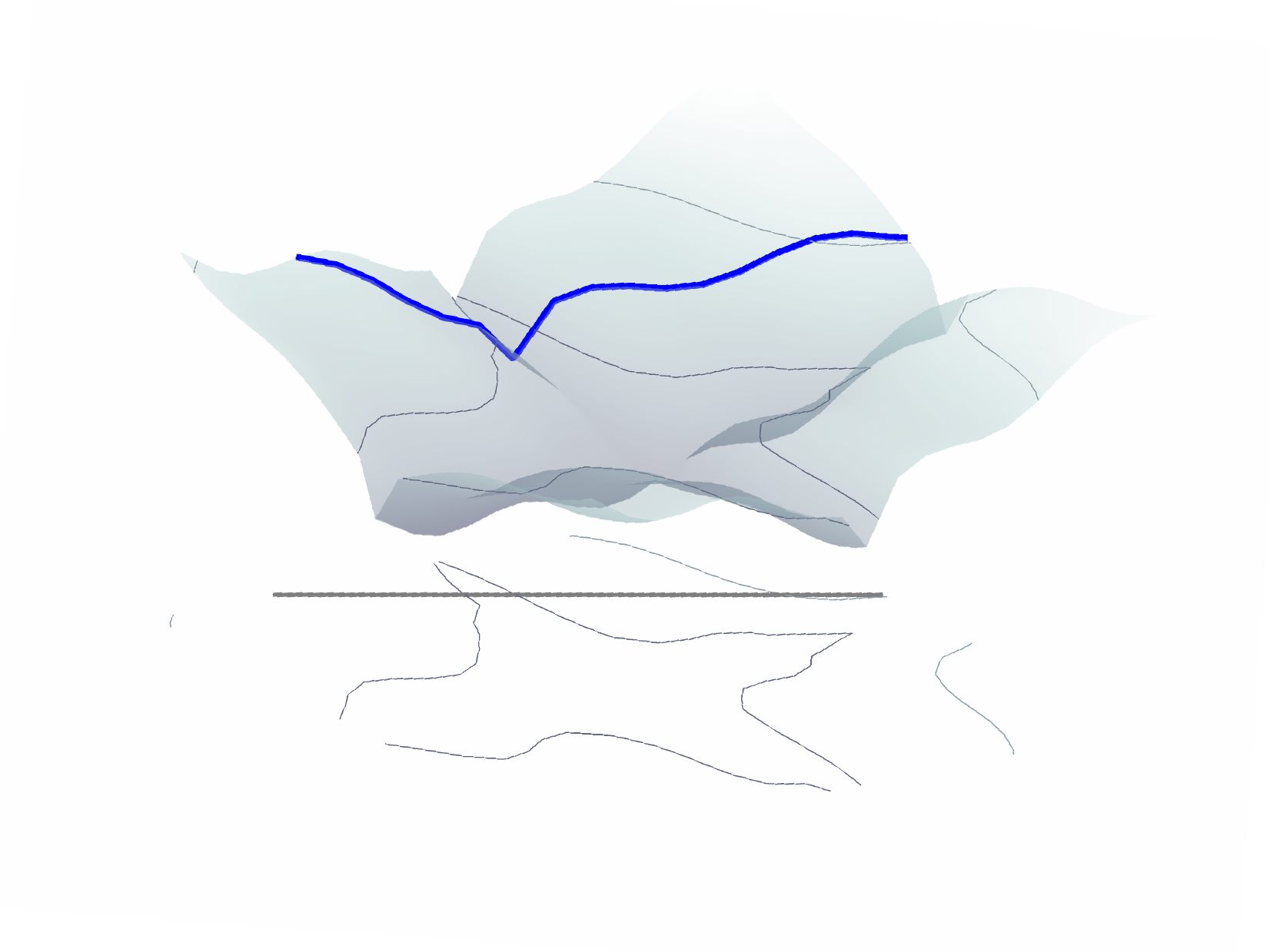}
\includegraphics[width=0.48\linewidth]{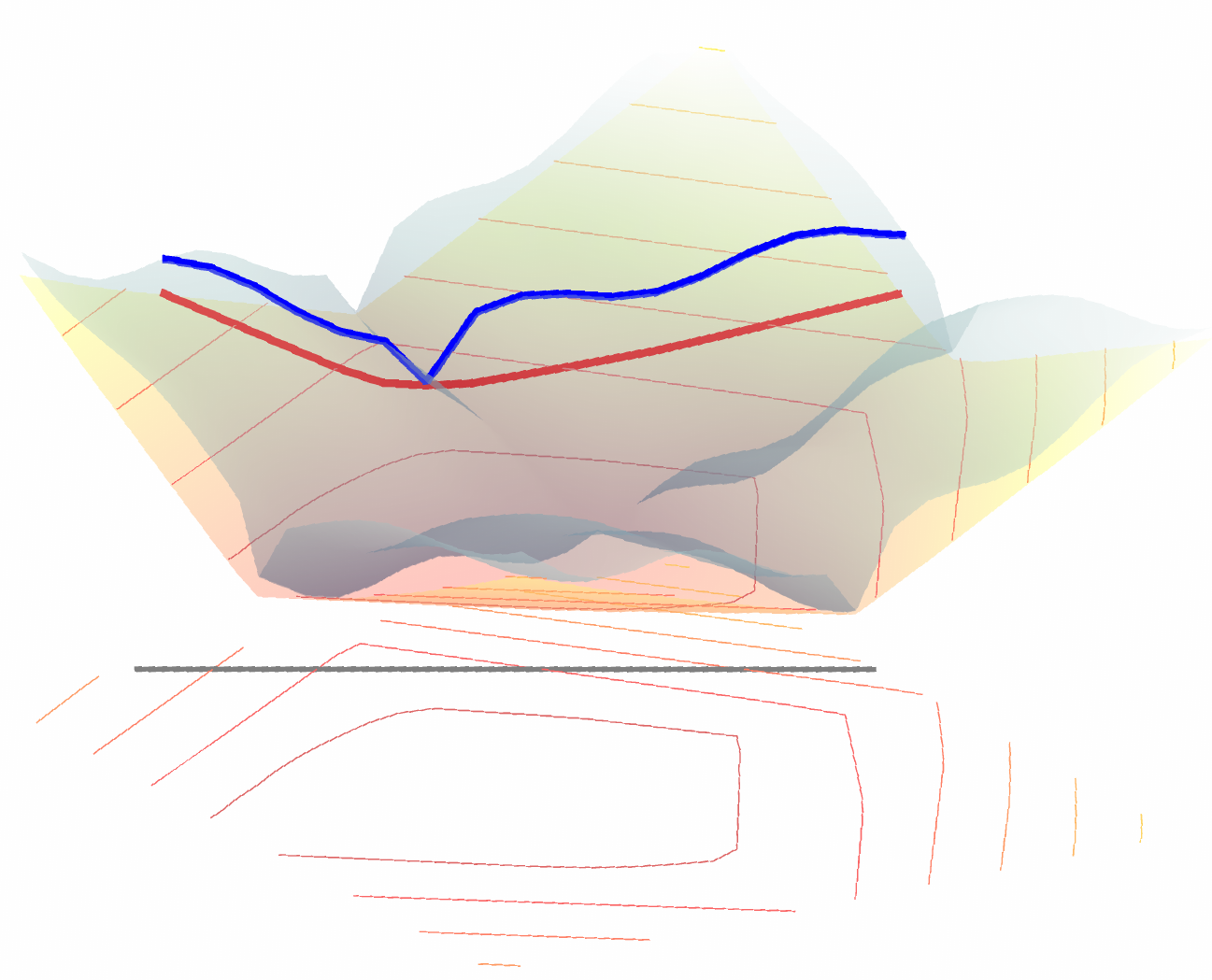}
\caption{Same setup as in Figure \ref{fig2intro}, but with a different subspace. The left graph illustrates \eqref{priorproblem1} (with $c\equiv 0$) and the right graph \eqref{restrictedproblem4}. Technically, $c$ is the indicator function of the subspace (think grey line), and the restriction of the respective functionals to this subspace is shown on top. We see that both \eqref{priorproblem1} and \eqref{restrictedproblem4} yield the same solution.}
\label{fig3}
\end{figure}

Let $c$ be a convex functional on $\V$ incorporating prior information known about the problem in question, and suppose we wish to minimize
\begin{equation}\label{priorproblem1}\argmin_{x} f(x)+\frac{1}{2}\|x-d\|^2+c(x).\end{equation}
For concrete examples of this form we refer e.g.~to the overview article \cite{tseng2010approximation}, or Section 4 of \cite{larsson2016convex} which contains applications to structure from motion and system identification. In particular, we can take as $c$ the indicator function of some closed convex subset $\H$ of $\V$, i.e. $c=\iota_{\H}$.

We suppose now that \eqref{priorproblem1} does not have a closed form solution, and consider replacing it by
\begin{equation}\label{restrictedproblem4}\argmin_{x} \S_\gamma^2(f)(x)+\frac{1}{2}\|x-d\|^2+c(x)\end{equation} to obtain a strongly convex problem (for $\gamma< 1$, just convex if $\gamma=1$). We warn that although $\S_1^2(f)(x)+\frac{1}{2}\|x-d\|^2$ is the l.s.c.~convex envelope of $f(x)+\frac{1}{2}\|x-d\|^2$ (by Theorem \ref{t1}), it is usually not true that the functional in \eqref{restrictedproblem4} is the l.s.c.~convex envelope of \eqref{priorproblem1}. Hence \eqref{restrictedproblem4} is a different problem with possibly a different answer. However, one of the key points of this section is that it often happens that they do have the same solution, and it is easy to see that this happens precisely when \begin{equation}\label{gre}f(\hat x)=\S_\gamma^2(f)(\hat x)\end{equation} holds for the solution $\hat x$ of \eqref{restrictedproblem4} (Proposition \ref{y66}) which is easily verified if a concrete expression for $\S^2_\gamma(f)$ is available. This is highlighted in Figure \ref{fig2intro} (where the two problems have a slightly different solution) and Figure \ref{fig3} (where the two problems have the same solution).

As mentioned in the introduction, the above suggested relaxation should be compared with ``traditional'' ones like for instance the nuclear norm, in the case where $f$ equals the rank of a matrix. For more details on the relation between original vs.~relaxed problem in this case, see e.g.~the discussion in \cite{tseng2010approximation}. The main conclusions of this section read as follows, (we refer to \cite{bauschke2011convex} for definitions of strongly convex and supercoercive).

\begin{proposition}\label{y66}
Let $f$ be a $[0,\infty]$-valued functional on a separable Hilbert space $\V$, and let $c\geq 0$ be a l.s.c.~convex function such that $\mathsf{dom} f\cap \mathsf{dom} c\neq \emptyset$. Given $\gamma<1$, the functional in \eqref{restrictedproblem4} is strongly convex and supercoercive. The solution is thus a unique point $\hat x$, which solves \eqref{priorproblem1} whenever $\S_{\gamma}^2(f)(\hat x)=f(\hat x)$.
\end{proposition}

\begin{proof}
By Theorem \ref{t1}, the functional $\S_{\gamma}^2(f)(x)+\frac{\gamma}{2}\|x-d\|^2+c(x)$ is l.s.c.~convex (but not necessarily strictly convex), and hence the functional in \eqref{restrictedproblem4} (obtained by adding $\frac{1-\gamma}{2}\|x-d\|^2$) is l.s.c.~and strongly convex. Supercoercivity is obvious due to the term $\frac{1}{2}\|x-d\|^2$ since $\S_{\gamma}^2(f)\geq 0$. Corollary 11.16 in \cite{bauschke2011convex} applied to the functional $\S_{\gamma}^2(f)(x)+\frac{1}{2}\|x-d\|^2+c(x)$, shows that \eqref{restrictedproblem4} has a unique minimizer $\hat x$. Since $f\geq \S_{\gamma}^2(f)$, it follows that $\hat x$ solves \eqref{priorproblem1} under the assumption that $\S_{\gamma}^2(f)(\hat x)=f(\hat x)$.
\end{proof}

\begin{ex}\label{tuesday}
Returning to Example \ref{exweights1}, suppose the so called ``model order'' $M$ is known, i.e.~we know beforehand the desired rank of the Hankel matrix sought. As in Example \ref{ale}, let $\mathcal{M}_M\subset\m_{n,n}$ be the manifold of matrices of rank $\leq M$ and let $\H$ be the linear subspace of all Hankel matrices. Given data $d$ which we want to approximate with at most $M$ exponential functions, a simple idea is to set $X_0=H_d$ and alternatingly project between $\H$ and $\mathcal{M}_M$ (this rationale is explain in \eqref{ale2} below). This goes back (at least) to \cite{cadzow1988signal}, and is sometimes known as Cadzow's algorithm. Local convergence results were first established in \cite{lewis2008alternating} and stronger results in the same spirit were given in \cite{andersson2013alternating}. In either case, if this converges there is no guarantee that it will converge to the optimal point, i.e. the solution to
\begin{equation}\label{ale1}\argmin_{X\in\H} \iota_{\mathcal{M}_M}(X)+\frac{1}{2}\|X-H_d\|_{\m_{n,n}}^2.\end{equation}
Based on Example \ref{ale} we can now compute the l.s.c. convex envelope of the functional in \eqref{ale1}, and apply convex optimization routines to find a global minimizer. As long as $X$ has distinct singular values, we have $\iota_{\mathcal{M}_M}(X)=\S_\gamma^2(\iota_{\mathcal{M}_M})(X)$ if and only if $\rank (X)\leq M$ (see Theorem 2 in \cite{andersson2016convex}), and hence a solution to the original problem is found if these conditions are met for the minimizer. Otherwise, the algorithm is still likely to yield a low rank approximation of the optimal point.

As a final remark, based on Kronecker's theorem and \eqref{otherweight}, one can show that \eqref{ale1} is equivalent (with the exception of some degenerate cases, c.f. \cite{andersson2016structure}) to the following:
\begin{equation}\label{ale2}\argmin_{f_j=\sum_{k=1}^M c_k e^{\zeta_k j}} \sum_{j=1}^{2n-1}\omega_j|f_j-d_j|^2\end{equation} for $\zeta_1,\ldots,\zeta_M\in\C$ and  $c_1,\ldots,c_M\in\C$, where $\omega_j$ is the triangle weight \eqref{triangleweight}. If we are interested in minimization over the standard (flat) $\ell^2$-norm, we may instead consider
\begin{equation}\label{ale3}\argmin_{X\in\H} \iota_{\mathcal{M}_M}(X)+\frac{1}{2}\|X-H_d\|_{\m_{n,n}^W}^2\end{equation} with $W$ as in Example \ref{exweights1}, which amounts to minimizing \eqref{ale2} with the weight seen to the right in Figure \ref{figweights}. With the same argument as in Example \ref{exweights1}, we have that the l.s.c. convex envelope of the functional in \eqref{ale3} is given by $$\S^2(\iota_{\mathcal{M}_M})(I_{\sqrt{u}}XI_{\sqrt{v}})+\frac{1}{2}\|X-H_d\|_{\m_{n,n}^W}^2.$$
\end{ex}

Several algorithms can be used to solve \eqref{restrictedproblem4}. One may use ADMM as in \cite{larsson2016convex}, or forward-backward splitting (see e.g.~\cite{combettes2005signal}). These algorithms have in common that one needs to be able to compute the proximal operator of $\S^2_\gamma(f)$, which we discuss next.


\subsection{The proximal operator}\label{prox}

To solve problem \eqref{restrictedproblem4} by either ADMM or FBS, we need to compute the proximal operator, i.e.
\begin{equation}\label{66}\prox_{\S_{\gamma}^2(f)/\rho}(y)=\argmin_{x} \S_{\gamma}^2(f)(x)+\frac{\rho}{2}\|x-y\|^2\end{equation}
for $\rho>\gamma$. Obviously, if one has a concrete expression for $\S_{\gamma}^2(f)$ it may be possible to compute \eqref{66} directly. However, we shall see in this section that \eqref{66} is computable even if we only have an expression for $\S_\gamma(f)$. In fact, even when both options are available, they may lead to different methods for the evaluation of $\prox$. A concrete example of this concerns the functional in Example \ref{exfixedK}, where $\S_{\gamma}(f)$ has a very simple expression and $\S_{\gamma}^2(f)$ has a very complicated one.

\begin{proposition}\label{propprox}
For $\rho>\gamma$ and $z=\prox_{\frac{\rho-\gamma}{\rho\gamma}\S_{\gamma}(f)}(y)$ we have
$$\prox_{\S_{\gamma}^2(f)/\rho}(y)=\frac{\rho y-\gamma z}{\rho-\gamma}.$$
\end{proposition}

\begin{proof}

The proof is a slight alteration of the classical Moreau decomposition (see \cite{parikh2014proximal} Sec. 2.5 or Theorem 14.3(ii) in \cite{bauschke2011convex}).
Since $$\argmin_x \S_{\gamma}^2(f)(x)+\frac{\rho}{2}\|x-y\|^2=\argmin_x\para{\S_{\gamma}^2(f)(x)+\frac{\gamma}{2}\|x\|^2}+\frac{\rho-\gamma}{2}\fro{x-\frac{\rho}{\rho-\gamma}y}$$
it follows that $\prox_{\S_{\gamma}^2(f)/\rho}(y)=\prox_{\para{\frac{1}{\rho-\gamma}\para{\S_{\gamma}^2(f)(x)+\frac{\gamma}{2}\|x\|^2}}}\para{\frac{\rho}{\rho-\gamma}y}$ which by the Moreau decomposition equals $$\frac{\rho}{\rho-\gamma}y-\frac{1}{\rho-\gamma}\prox_{(\rho-\gamma)\para{\S_{\gamma}^2(f)(x)+\frac{\gamma}{2}\|x\|^2}^*}(\rho y),$$ so it suffices to show that the latter proximal operator equals $\gamma z$. Note that
$$(\rho-\gamma)\para{\S_{\gamma}^2(f)(x)+\frac{\gamma}{2}\|x\|^2}^*(\cdot)=(\rho-\gamma)\para{\S_\gamma({f})(\frac{x}{\gamma})+\frac{\gamma}{2}\|\frac{x}{\gamma}\|^2}$$ by Theorem \ref{t1} (applied with $d=0$). Using the identity $\prox_{g\para{\frac{\cdot}{\gamma}}}(y)=\gamma \prox_{\frac{f}{\gamma^2}}\para{\frac{y}{\gamma}}$ the proximal operator $\prox_{(\rho-\gamma)\para{\S_{\gamma}^2(f)(x)+\frac{\gamma}{2}\|x\|^2}^*}(\rho y)$ becomes \begin{align*}&\gamma \prox_{\frac{\rho-\gamma}{\gamma^2}\para{\S_\gamma({f})(x)+\frac{\gamma}{2}\|x\|^2}}\para{\frac{\rho y}{\gamma} }=\gamma\argmin_x \frac{\rho-\gamma}{\gamma^2}\para{\S_\gamma({f})(x)+\frac{\gamma}{2}\fro{x}}+\frac{1}{2}\fro{x-\frac{\rho y}{\gamma}}=\\&
\gamma\argmin_x \frac{\rho-\gamma}{\gamma^2}\S_\gamma({f})(x)+\frac{\rho}{2\gamma}\fro{x}-\frac{\rho}{\gamma}\scal{x,y}=\gamma\argmin_x \frac{\rho-\gamma}{\gamma^2}\S_\gamma(f)(x)+\frac{\rho}{2\gamma}\fro{x-y}=\\&\gamma\argmin_x \frac{\rho-\gamma}{\rho\gamma}\S_\gamma({f})(x)+\frac{1}{2}\fro{x-y}=\gamma\prox_{ \frac{\rho-\gamma}{\rho\gamma}\S_\gamma(f)}=\gamma z,\end{align*}
as desired.

\end{proof}

\section{Part III; quadratic terms of the form $\|Ax-d\|^2$}
\subsection{Motivation and examples}\label{secex2}
$\S(f)$ is explicitly computable whenever \begin{equation*}\label{gt}\argmin_{x} f(x)+\frac{1}{2}\|x-d\|^2\end{equation*} has explicit solutions for all $d\in\V$, making the unconstrained problem rather uninteresting. However, for the problem
\begin{equation}\label{17}\|x\|_0+\frac{1}{2}\fro{Ax-d}_2,\quad x\in\C^n,\end{equation}
the key objective is simply finding the global minimizer. The remainder of the paper is devoted to the study of such cases. We henceforth consider \begin{equation}\label{t4}\J(x)=f(x)+\frac{1}{2}\fro{A x-d}_\W,\quad x\in \V\end{equation}
where $\V,\W$ are possibly different (separable) Hilbert spaces and $A:\V\rightarrow\W$ is linear and bounded. 
We point out that, in case $A$ is bounded from below, we may introduce a new Hilbert space ${\V_A}$, which equals $\V$ as a vector space but with the new norm $\|x\|_{{\V_A}}=\|Ax\|_{\W}$, and then ``compute'' the l.s.c. convex envelope of \eqref{17} by applying $S_{\V_A}$ twice to $f$. In case $A$ is not bounded from below, ${\V_A}$ is only a semi-normed space which may not be complete, but we could still develop a theory similar to that in Part I. However, the problem arise since $S_{{\V_A}}(f)$ usually has no explicit formula, and hence the theory becomes vacuous. Moreover, when $A$ has a kernel, $f$ is bounded and $f(0)=0$, it is easy to see that $S_{{\V_A}}^2(f)(x)= 0$ for all $x$ in the kernel of $A$, and hence the convex envelope is not a desirable functional for solving e.g. \eqref{17}. In the particular case of problem \eqref{17}, a very interesting idea to cope with this problem is suggested in \cite{selesnick2017sparse}.

Our aim here is to develop strategies to deal with the general problem \eqref{t4}, in the case when $f$ is an $[0,\infty]$-valued functional such that $S_{\V,\gamma}(f)=\S_\gamma(f)$ is computable, and focus on computing (explicit) approximations of the l.s.c convex envelope of $\J$. The remaining theory is split in two cases, either we approximate the convex envelope from below by a convex functional, or we approximate it from above with a non-convex functional having a number of desirable properties, most notably continuity and the fact that local minimizers do not change. More precisely, we will study the relationship between the original problem \eqref{t4} and the modified problem
\begin{equation}\label{t4mod}\J_\gamma(x)=\S_{\gamma}^2(f)(x)+\frac{1}{2}\fro{A x-d}_\W,\quad x\in \V\end{equation}
under the assumption that $\gamma I\geq A^*A$ (case 1) or $\gamma I\leq A^*A$ (case 2). We now provide one example which highlight the two possibilities of choosing $\gamma$.

\begin{ex} Let $W\in\m_{m,n}$ be strictly positive and recall that $\m_{m,n}^W$ is equipped with the norm $$\fro{X}_{W}=\sum_{i,j}w_{i,j}|x_{i,j}|^2$$ (see the text preceding Example \ref{exmatrixes}). Suppose we are interested in the l.s.c. convex envelope of the non-convex functional \begin{equation}\label{rank2}\rank(X)+\frac{1}{2}\fro{X-D}_W.\end{equation} Note that \eqref{rank2} can be written as \begin{equation}\label{rank25}\rank(X)+\frac{1}{2}\fro{A(X)-A(D)}_F\end{equation} where $A$ is the linear operator on $\m_{m,n}$ of pointwise multiplication with $\sqrt{W_{i,j}}$. From Example \ref{exmatrixcont} we know that the l.s.c convex envelope has a closed form expression in the special case when $W$ is a direct tensor, but the majority of weights $W$ are clearly not of this form. So we assume that $W$ is not a direct tensor and hence no explicit formula for the l.s.c. convex envelope of \eqref{rank2} is available. We thus have to satisfy with estimates of the desired l.s.c. convex envelope. Consider
\begin{equation}\label{rankup}\S_{\m_{m,n},\gamma}^2(\rank)(X)+\frac{1}{2}\fro{X-D}_W,\end{equation} where the transforms in the above formula have explicit expressions by formula \eqref{SSrank}.

If we suppose that $\gamma\leq \underline{\gamma}=\min_{i,j}{W_{i,j}}$ we shall show that \eqref{rankup} is convex, whereas if $\gamma\geq \overline{\gamma}=\max_{i,j}{W_{i,j}}$ the minimizers of \eqref{rankup} are the same as those of the original functional \eqref{rank2}. Moreover, the l.s.c. convex envelope of \eqref{rank2} sits in between the two possibilities, i.e. (omitting the explicit reference to $\m_{m,n}$ for easy reading)
$$\S_{\underline{\gamma}}^2(\rank)(X)+\frac{1}{2}\fro{X-D}_W\leq CE\para{\rank(X)+\frac{1}{2}\fro{X-D}_W}\leq \S_{\overline{\gamma}}^2(\rank)(X)+\frac{1}{2}\fro{X-D}_W.$$

As a method for ``solving'' \eqref{rank2}, replacing it with either of the two possibilities may seem ad hoc but we remind the reader that minimization of the convex problem
\begin{equation}\label{nuclear}\|X\|_*+\frac{1}{2}\fro{X-D}_W\end{equation} where $\|X\|_*$ denotes the nuclear norm,
has become very popular in recent years, where the rationale behind considering \eqref{nuclear} instead of \eqref{rank2} is that the nuclear norm appears as the convex envelope of the rank restricted to the unit ball. Clearly, both options considered here stay closer to the original problem \eqref{rank2} than \eqref{nuclear}.
\end{ex}

We end this section with a concrete toy-example providing intuition for the two possibilities, which despite its simplicity summarize the general picture. Recall that $|\cdot|_0$ is the characteristic function of $\R\setminus\{0\}$.

\begin{figure}
\begin{center}
\includegraphics[width=0.49\linewidth]{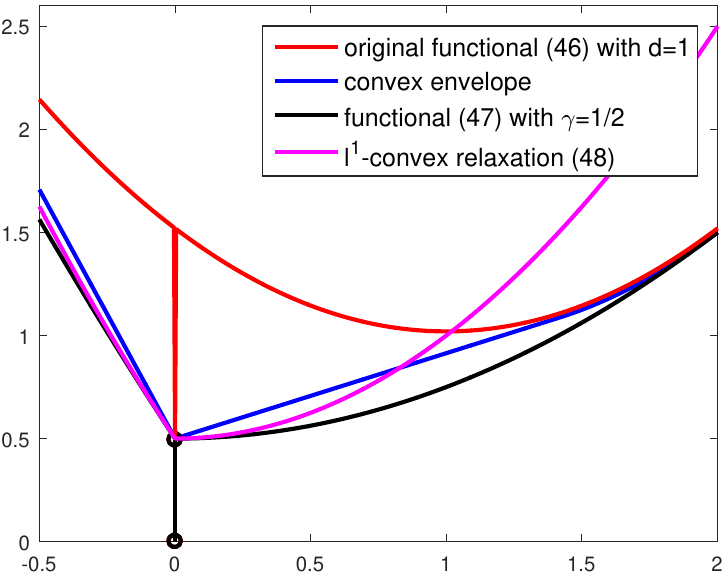}
\includegraphics[width=0.49\linewidth]{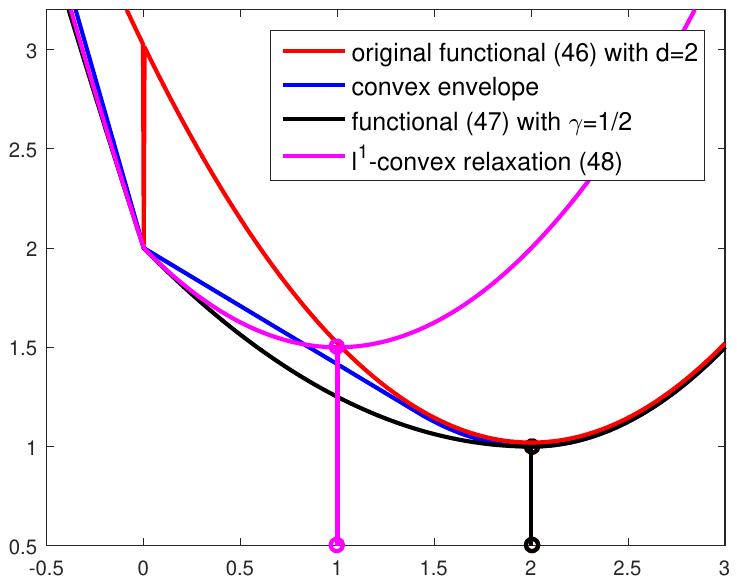}
\end{center}
\caption{Illustration to Example \ref{underestimate}. The circles show location of global minimizers. In both cases, the minimizer of \eqref{rankdown1} and \eqref{rankdown2} coincide, although this is not always the case (see Figure \ref{fig3intro}).}\label{fig1}
\end{figure}

\begin{ex}\label{underestimate}
Let $\V=\R$ and suppose we wish to minimize \begin{equation}\label{rankdown1}|x|_0+\frac{1}{2}\fro{x-d}=|x|_0+\frac{1}{2}|x-d|^2\end{equation} with  $d=1$, i.e. the red curve in the left graph of Fig. \ref{fig1}. As is readily seen, the minimum occurs at $x=0$ which is also the unique minimum of the l.s.c. convex envelope, painted in blue. It differs from \eqref{rankdown1} in the interval $[-\sqrt{2},\sqrt{2}]$ where it is piecewise linear.

However, suppose for the sake of the argument that this minimum, as well as its l.s.c. convex envelope, are impossible to compute analytically. In analogy with \eqref{rankup} we thus replace \eqref{rankdown1} by
\begin{equation}\label{rankdown2}\S_{\gamma}^2(|\cdot|_0)(x)+\frac{1}{2}|x-d|^2,\end{equation}
where $\S_\gamma^2(|\cdot|_0)(x)=1-\para{\max\{1-\frac{\sqrt{\gamma}|x|}{\sqrt{2}},0\}}^2,$ which follows by \eqref{15}. For $\gamma=1/2$, \eqref{rankdown2} is depicted in black. We note that \eqref{rankdown2} is convex and also have $x=0$ as the global minimizer. This is however not always the case, as Figure \ref{fig3intro} shows. We also include a graph of the expression \eqref{nuclear} adapted to this situation, i.e. \begin{equation}\label{nuclear1}|x|+\frac{1}{2}|x-d|^2,\end{equation} which also has $x=0$ as minimizer.

The right graph shows the same setup but with $d=2$. In this case the global minimizer of \eqref{rankdown1} and \eqref{rankdown2} coincide, but the minimizer of \eqref{nuclear1} is different. In fact we see the well known shrinking bias pertinent to these techniques.
\end{ex}

\begin{ex}\label{overestimate}
This example is the same as the previous except $\gamma=2$, and is illustrated in Figure \ref{fig2}. We note that \eqref{rankdown2} is non-convex (but at least continuous) and moreover neither the global nor local minimizers have moved. We also see that the amount of local minimizers of \eqref{rankdown2} may be fewer than those of \eqref{rankdown1} (and in real scenarios often is, see \cite{soubies2015continuous}). We prove this statement in a more general setting in Section \ref{secoverestimate}. For a case where \eqref{rankdown2} has precisely the same local minimizers as \eqref{rankdown1}, see Figure \ref{fig3intro}.
\end{ex}

\begin{figure}
\begin{center}
\includegraphics[width=0.49\linewidth]{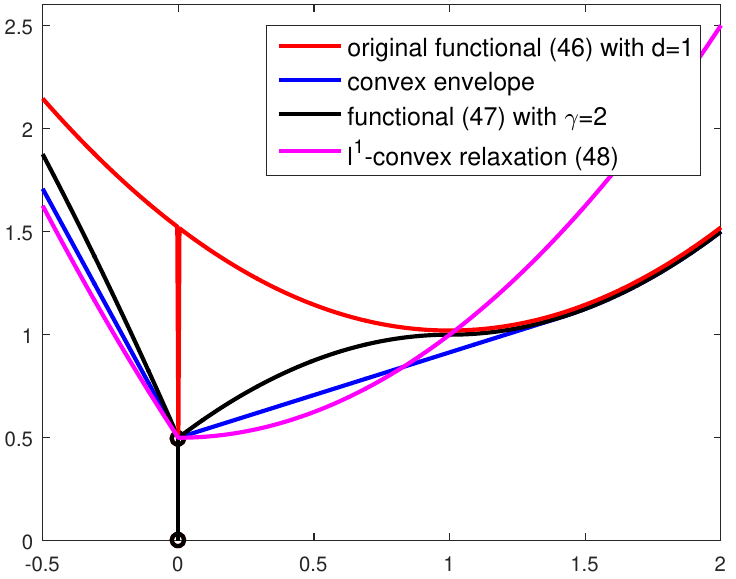}
\includegraphics[width=0.49\linewidth]{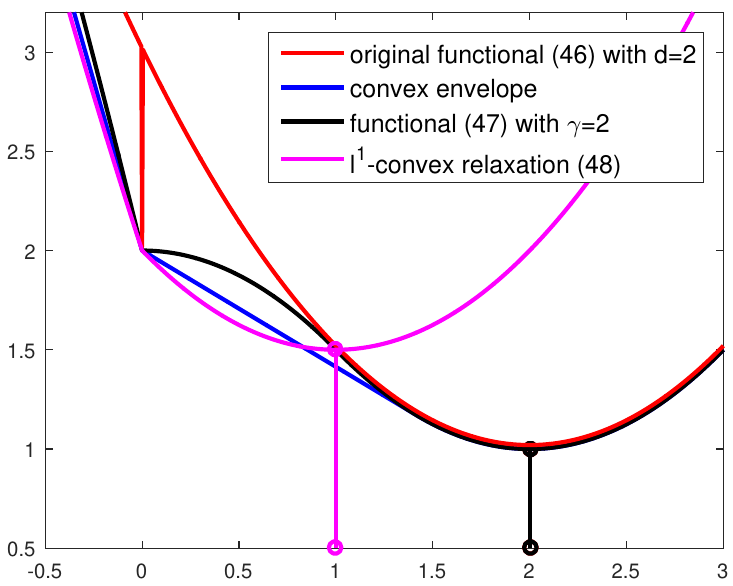}
\end{center}
\caption{Illustration to Example \ref{overestimate}. Circles represent global minima.}\label{fig2}
\end{figure}

In the coming sections we shall see that the situation in Examples \ref{underestimate}-\ref{overestimate} is symptomatic for the general case.

\subsection{Case $A^*A\geq \gamma I$.}\label{secunderestimate}
Let $f$ be a $[0,\infty]-$valued functional and $A:\V\rightarrow\W$ a bounded linear operator. We remind the reader that we are interested in the relationship between $\J$ and $\J_\gamma$ defined in \eqref{t4} and \eqref{t4mod} respectively. The main result of this section is that $\J_\gamma$ is a convex minorant of the l.s.c. convex envelope of $\J$, denoted $CE(\J)$.

\begin{theorem}\label{tunder}
For $\gamma>0$ such that $A^*A\geq \gamma I$, $\J_\gamma$ is convex and $\J_\gamma\leq CE(\J)$. Moreover, if $A^*A> \gamma I$ then it is strongly convex, in which case it has a unique minimizer. Finally, a minimizer $\hat x$ of $\J_{\gamma}$ is a minimizer of $\J$ whenever $f(\hat x)=S^2_{\gamma}(f)(\hat x)$.
\end{theorem}
\begin{proof}
Upon expanding $\fro{Ax-d}=\|Ax\|^2+2\scal{Ax,d}+\fro{d}$ and noting that the latter two terms is affine linear, it is easily seen that it suffices to prove the first part of the proposition for $d=0$. That $\J_\gamma$ is l.s.c. and that $\J_\gamma\leq \J$ follows immediately by Theorem \ref{t1}, and thus $\J_\gamma\leq CE(\J)$ follows immediately upon showing that $\J_\gamma$ is convex. Define $\scal{x,y}_{\U}=\scal{Ax,Ay}_{\W}-\gamma\scal{x,y}_{\V}$ and note that this is a semi-inner product as long as $A^*A\geq \gamma I$, which is an inner product if the inequality is strict. In either case, $\fro{x}_{\U}:=\scal{x,x}_{\U}$ is convex (see Ch.~I.1 of \cite{conway2013course}). It follows that $$f(x)+\frac{1}{2}\fro{Ax}_{\W}=(f(x)+\frac{\gamma}{2}\fro{x}_{\V})+\frac{1}{2}\fro{x}_{\U},$$ which by Theorem \ref{t1} implies that $\J_\gamma$ equals the l.s.c. convex envelope of $f(x)+\frac{\gamma}{2}\fro{x}_{\V}$ plus the term $\frac{1}{2}\fro{x}_{\U}$. We conclude that $\J_\gamma$ is a convex functional, which is strongly convex when $A^*A> \gamma I$. In the latter case, the existence of a unique minimizer follows by Corollary 11.15 in \cite{bauschke2011convex}, (supercoercivity of $\J_\gamma$ is obvious by the term $\frac{1}{2}\fro{x}_{\U}$).

Now let $d$ be fixed and let $\hat x$ be a minimizer of $\J_\gamma$. Suppose that $f(\hat x)=\S_\gamma^2(f)(\hat x)$ and let $y\in\V$ be arbitrary. Then $\J(y)\geq \J_\gamma(y)\geq \J_\gamma(\hat x)=\J(\hat x)$, showing that $\hat x$ is a global minimizer of $\J$.
\end{proof}

\subsection{Case $A^*A\leq \gamma I$.}\label{secoverestimate}

Let $f$ be a $[0,\infty]-$valued functional and $A:\V\rightarrow\W$ a bounded linear operator. Again we are interested in the relationship between $\J$ and $\J_\gamma$ defined in \eqref{t4} and \eqref{t4mod} respectively. The main result of this section is that $\J_\gamma$ does not move minima for $\gamma$ in the stated range, but we begin by noting the following inequalities, the first one being reverse of the one proved in Theorem \ref{tunder}.

\begin{proposition}\label{pover}
For $\gamma$ such that $A^*A\leq \gamma I$, we have \begin{equation}\label{ineq2}CE(\J)\leq \J_\gamma\leq \J.\end{equation}
\end{proposition}
\begin{proof}
The right inequality is immediate since $\S^2_\gamma(f)\leq f$ by Theorem \ref{t1}. As in Theorem \ref{tunder} we moreover see that it suffices to prove the left inequality for $d=0$. To this end, set $h(x)=CE(\J)(x)-\frac{1}{2}\|Ax\|^2$. Since $CE(\J)\leq f+\frac{1}{2}\fro{Ax}$ we have $h\leq f$ and moreover $$h(x)+\frac{\gamma}{2}\fro{x}=CE(\J)+\para{\frac{\gamma}{2}\fro{x}-\frac{1}{2}\fro{Ax}}.$$ The right hand side is convex and l.s.c., by which we conclude that $$h(x)+\frac{\gamma}{2}\fro{x}\leq CE(f+\frac{\gamma}{2}\fro{\cdot})(x)=\S_\gamma^2(f))(x)+\frac{\gamma}{2}\fro{x}$$
(the last identity follows by Theorem \ref{t1}, which gives $h(x)\leq \S_\gamma^2(f)(x)$. In other words $CE(\J)(x)\leq \S_\gamma^2(f)(x)+\frac{1}{2}\|Ax\|^2$, which is the desired inequality (for $d=0$).
\end{proof}

We now come to the main theorem of this section, inspired by Theorems 4.5 and 4.8 in \cite{soubies2015continuous}.  We say that $x$ is a local minimizer of $\J$ if there exists a neighborhood $U$ of $x$ in $\V$ such that $\J(y)\geq \J(x)$ for all $y\in U$ and we say that $x$ is a strict local minimizer of $\J$ if the inequality is strict for $y\neq x$.

\begin{theorem}\label{tover1}
Suppose that $A^*A< \gamma I$. If $x$ is a local minimizer (resp. strict local minimizer) of $\J_\gamma$, then it is also a local minimizer (resp. strict local minimizer) of $\J$, and $\J_\gamma(x)=\J(x)$. In particular the global minimizers coincide.
\end{theorem}
\begin{proof}
Let $x$ be a local minimizer of $\J_\gamma$. If \begin{equation}\label{hyf}\S_\gamma^2(f)(x)=f(x)\end{equation} does not hold, then Theorem \ref{propconcave} implies that there exists a unit vector $\nu$ such that $$t\mapsto\para{\S_\gamma^2(f)(x+t\nu)+\frac{\gamma}{2}\fro{x+t\nu}_{\V}}$$
is affine near $t=0$. Introducing $\fro{x}_\U=\gamma\fro{x}_\V-\fro{Ax}_\W$, we have (as in the proof of Theorem \ref{tunder}) that $\|\cdot\|_\U$ defines a norm. Note that \begin{equation}\label{brysel}\begin{aligned}&\J_\gamma(x+t\nu)=\S_\gamma^2(f)(x+t\nu)+\frac{1}{2}\fro{A(x+t\nu)-d}_{\W}=\\&\para{\S_\gamma^2(f)(x+t\nu)
+\frac{\gamma}{2}\fro{x+t\nu}_{\V}}-\frac{1}{2}\|x+t\nu\|^2_\U- \scal{A(x+t\nu),d}_\W+\frac{1}{2}\|d\|^2_\W,\end{aligned}\end{equation} whose second derivative equals $-\fro{\nu}_{\U}<0$ at $t=0$, contradicting the assumption that $x$ is a local minimizer of $\J_\gamma$ (the inequality is strict since $\|A\|^2<\gamma$). We thus conclude that \eqref{hyf} holds, i.e. that $\J_\gamma(x)=\J(x)$. In view of Proposition \ref{pover}, it follows that $x$ is a local minimizer also for $\J$. The same argument applies to strict local minimizers.

We now prove that the global minimizers coincide. Note that global minimizers of $\J$ are global minimizers of $\J_\gamma$ in view of \eqref{ineq2} and the fact that $\J(x)=CE(\J)(x)$ for all global minimizers $x$. From this we also see that the global minimum of $\J$ and $\J_\gamma$ coincide, let us denote this value by $c$. Conversely suppose that $x$ is a global minimizer of $\J_\gamma$, i.e. $\J_\gamma(x)=c$. Then it is a local minimizer of $\J$ by the first part, which automatically is global for $\J$ since we otherwise would have $\J(y)<c$ for some other value $y$. The proof is complete.
\end{proof}

The situation when $A^*A= \gamma I$ (i.e. $\gamma=\|A\|^2$) is a bit more involved, so we content ourselves with the following statement concerning the global minimizers.

\begin{theorem}\label{tover2}
Set $\gamma=\|A\|^2$, let $G$ be the global minimizers of $\J$ and $G_\gamma$ the global minimizers of $\J_\gamma$. Then $G\subset G_\gamma$, and each connected component of $G_\gamma$ contain points of $G$.
\end{theorem}
\begin{proof}
The statement $G\subset G_\gamma$ follows as in the above proof, as well as the fact that the global minimum of $\J$ and $\J_\gamma$ coincide; we denote it by $c$.

If $x\in G_\gamma$ and $\J(x)> c$, then it follows by \eqref{brysel} that there exists a unit vector $\nu$ such that $\frac{d^2}{dt^2}\J_\gamma(x+t\nu)\leq 0$ in a neighborhood of $t=0$. Strict inequality contradicts the assumption of global minima, so we deduce that $\J_\gamma(x+t\nu)$ is constant near $t=0,$ and hence \eqref{brysel} yields $\|\nu\|_{\U}=0$, i.e. that $\nu$ lies in the kernel of the semi-norm $\|\cdot\|_{\U}$ (which is a linear subspace by convexity of the semi-norm). Let $P$ be the affine hyperplane $P=x+\ker \|\cdot\|_{\U}$ and set $S=P\cap G_\gamma.$ For $y\in \ker \|\cdot\|_{\U}$, \eqref{brysel} implies that \begin{equation}\label{t6}\J_\gamma(x+y)=\para{S_{\gamma}^2(f)(x+y)+\frac{\gamma}{2}\fro{x+y}_{\V}}-\fro{x}_{\U}- \scal{A(x+y),d}_\W+\frac{1}{2}\|d\|^2_\W,\end{equation} so Theorem \ref{t1} implies that $\J_\gamma$ is convex on $P$. In particular, $S$ is convex. Since $\J_\gamma$ is l.s.c. it is also closed. Moreover $S$ is bounded due to the quadratic term $\fro{x+y}_{\V}$ in \eqref{t6}. $S$ is therefore weakly closed, and hence it equals the closed convex hull of its extremal points, by the Krein-Milman theorem. If $x$ now is one of these extremal points, then we can argue as in the beginning of this proof and conclude that $\J_\gamma(x)=\J(x)$, since the existence of a $\nu$ with the properties stated initially would contradict that $x$ is an extremal point of $S$.
\end{proof}

Methods for minimizing non-convex functionals include \cite{attouch2013convergence,beck2015convergence,bolte2014proximal,gong2013general,hare2009computing,ochs2015iteratively,yu2015minimizing}. As earlier in this paper, we do not discuss which algorithms are more suitable, but note that all can get stuck in local minima due to the non-convexity.

\subsection{Remarks on $CE(\J)$}

We include some theoretical results concerning $CE(\J)$, disregarding the computability aspect. Define $\V_A$ to be $\V$ equipped with the semi-norm $\|x\|_{\V_A}=\|Ax\|_\W$. This makes $\V_A$ into a semi-normed space (see e.g. Ch I.1 in \cite{conway2013course}). The definition of the $\S$ transform \eqref{defS} is readily extended to this situation,
\begin{equation}\label{defSVA}\S_{\V_A}(f)(y):=\sup_x -f(x)-\frac{1}{2}\fro{A(x-y)}_{\V},\end{equation} and many of the results of part I can be generalized to include this case. In particular we have

\begin{proposition}\label{pl0} Let $\V$ be a finite dimensional space. Then $$CE(\J)=\S_{\V_A}^2(f)+\frac{1}{2}\|Ax-d\|^2_\V.$$\end{proposition}
\begin{proof}
$\|\cdot\|_{\V_A}$ is a norm on the vector space $(\ker A)^{\perp}$, which becomes a Hilbert space since it clearly is generated by a scalar product and $\V$ is finite dimensional. We denote this Hilbert space by $\U$. By the equivalence of all norms in finite dimensional spaces, we have that a function on $\U$ is l.s.c with respect to $\|\cdot\|_{\V_A}$ if and only if it is l.s.c. with respect to $\|\cdot\|_\V$. Set $$\tilde f(x)=\inf_{y\in \ker A} f(x+y), \quad x\in\U.$$ Since $\U$ is a Hilbert space we can apply the results from part I there. It is easy to see that both sides of the sought identity are constant on $\ker A$. More precisely, for $x\in \U$ and $y\in \U^{\perp}$ we have $$CE(\J)(x+y)=CE\left(\tilde f(\cdot)+\frac{1}{2}\|A\cdot -d\|_\W^2\right)(x)$$ and $\S_{\V_A}^2(f)(x+y)=\S_{\U}^2(\tilde f)(x)$, so it suffices to prove \begin{equation}\label{alingsas}CE\left(\tilde f(\cdot)+\frac{1}{2}\|A\cdot -d\|_\W^2\right)(x)=\S_{\U}^2(\tilde f)(x)+\frac{1}{2}\|Ax-d\|^2_\W, \quad x\in\U.\end{equation}
Since \begin{equation*} \frac{1}{2}\|Ax-d\|^2_\W=\frac{1}{2}\fro{x}_{\U}-\scal{Ax,d}_\W+\frac{1}{2}\fro{d}_\W\end{equation*} and the latter two terms are affine linear, they can be moved outside the parenthesis of $CE$ in \eqref{alingsas}, whereby they cancel the corresponding terms from the right hand side. It follows that it suffices to prove \eqref{alingsas} with $d=0$. But this is immediate by Theorem \ref{t1}.
\end{proof}

\section{Conclusions}

In Part I we have provided theory for computing l.s.c.~convex envelopes of functionals of the type $f(x)+\frac{1}{2}\fro{x-d}$, and shown a connection with Lasry-Lions approximants. These results and connections are new (to our best knowledge) unless explicitly stated. A number of prior works most notably by Carl Olsson and Viktor Larsson has paved the way for these insights.

In Part II we showed how the convex envelopes can be used in combination with additional restrictions, which is one of the key applications. After all, if the convex envelope is computable then the unconstrained minimization problem has an explicit solution, so the convex envelope becomes obsolete. This part is more for illustration and mainly contain ideas which has appeared elsewhere, albeit in case specific circumstances.

In Part III we considered unconstrained problems where the convex envelope is not computable due to a matrix $A$ in the quadratic term. We studied what happens if we replace $f$ by the nicer functional $\S_\gamma^2(f)$. We showed that for sufficiently small $\gamma$ this yields convex functions below the original functional, which coincide with the original functional on a large part of the underlying Hilbert space. For $\gamma$ sufficiently large on the other hand, we loose convexity but gain the desirable feature that the modified functional has the same minimizers as the original one. The results in this part are completely new to this paper, albeit some ideas have been generalized from those developed by Emmanuel Soubies, Laure Blanc-F\'{e}raud and Gilles Aubert.

\section{Related works}
The present work is the extension of a chain of ideas. $\ell^1-\ell^2$-minimization tricks have a long history and got renewed attention with the work of Donoho and Cand\'{e}s among others. In the same spirit the nuclear norm minimization strategy was investigated by Fazel and coworkers. These two trends were subsequently refined, in the particular case of a quadratic misfit term, by Carl Olsson and coworkers as well as by Blanc-Feraud and coworkers. The contribution of the present paper is a unifying framework, theoretical development and tools for computing the convex envelopes.

We have not found any similar result in the literature, my apologies if I have missed something. For completeness I include a list of results which seem relevant. The fairly recent survey paper \cite{lucet2010shape} is about the closely related concept of computing Fenchel conjugates, and also mentions the Lasry-Lions approximants, yet it has no overlap with the present paper despite citing 262 other papers. It primarily deals with numeric computation of convex envelopes in cases when symbolic formulas are not available, and as such it is an interesting alternative to the methods developed in Part III. The same goes for the papers \cite{mccormick1976computability} and \cite{borwein2009symbolic}, although the focus there is on symbolic computations using Maple. The papers \cite{attouch1993approximation,stromberg1996regularization} deal with Lasry-Lions approximants in Hilbert space, but does not make the connection with the convex envelopes. The importance of computing convex envelopes is stressed in \cite{meyer2005convex}, where techniques for computing convex envelopes of so called ``convex polyhedral'' functions is developed. Convex approximations from below are considered in \cite{brighi1994approximated}, which should be compared with the results in Section \ref{secunderestimate}. Other non-convex variations of the compressed sensing type techniques involve replacing the $\ell^1$-term with an $\ell^p-$term for $p<1$, see e.g. \cite{grasmair2010non,saab2010sparse}. From an algorithmic point of view, the methods in Part III boils down to analytically computing proximal operators for approximations of the convex envelope of the original functional to be minimized. An alternative is to numerically try to compute the proximal operator of the original functional, which is pursued in \cite{hare2009computing}.

\section*{Acknowledgment} I am grateful to Fredrik Andersson, J\'{e}r\^{o}me Bolte and Laure Blanc-F\'{e}raud for interesting discussions on this subject.

\section*{Appendix I; von-Neumann's trace inequality for operators}\label{ap2}
We repeat the statement of von-Neumann's trace inequality for operators on separable Hilbert spaces, which reads as follows.
\begin{theoremextra1}
Let $\V_1,~\V_2$ be any separable Hilbert spaces, let $X,Y\in \B_2(\V_1,\V_2)$ be arbitrary and denote their singular values by $\sigma_j(X)$, $\sigma_j(Y)$, respectively. Then \begin{equation}\label{vn1}\scal{X,Y}\leq \sum_j \sigma_j(X)\sigma_j(Y)\end{equation}
with equality if and only if the singular vectors can be chosen identically.
\end{theoremextra1}
Surprisingly, this statement is nowhere to be found in the standard references, such as the books by Simon \cite{simon1979trace}, Conway \cite{conway2000course}, Ringrose \cite{ringrose}, Schatten \cite{schatten2013norm}, Dowson \cite{dowson1978spectral} Dunford and Schwartz \cite{dunford1963linear}. The weaker inequality $\scal{X,Y}\leq \|X\|\|Y\|$, which is needed to show that the Hilbert-Schmidt class is a Hilbert space, is of course found in many of the above references, see e.g. Theorem 18, Section XI 6 of \cite{dunford1963linear}. In fact, the ``only if'' part of the statement is hard to find even in the finite dimensional case. It does appear in \cite{de1994exposed} (cited 15 times at the time of writing) along with a discussion claiming that even von-Neumann himself had this part of the statement clear. We take the finite dimensional part of the statement for granted and proceed to prove the infinite dimensional case.

\begin{proof}
If both $X$ and $Y$ have finite rank, we may consider $X$ and $Y$ to be operators between the finite dimensional spaces $(\mathsf{Ker}X\cup\mathsf{Ker} Y)^{\perp}$ and $\mathsf{Span}(\ran X\cup\ran Y)$. Then $X$ and $Y$ can be represented by matrices (upon choosing some orthonormal bases) and hence the finite dimensional version of the theorem applies. We can thus assume that \eqref{vn1} holds for finite rank matrices. 

Given any $Z\in\B_2$, let $(u_{Z,j})_{j=1}^\infty$ and $(v_{Z,j})_{j=1}^\infty$ be singular vectors, i.e. such that $$Z=\sum_{j=1}^\infty\sigma_j(Z)u_{Z,j}\otimes v_{Z,j}$$
where $\sigma_j(Z)$ are the singular values and $u_{Z,j}\otimes v_{Z,j}(x)=u_{Z,j}\scal{x, v_{Z,j}}$ (see e.g. Theorem 1.4 \cite{simon1979trace}).
Set $Z^J=\sum_{j=1}^J\sigma_j(Z)u_{Z,j}\otimes v_{Z,j}$ and note that this has finite rank. Thus
\begin{equation}\label{vn2}\scal{X,Y}=\lim_{J\rightarrow\infty} \scal{X^J,Y^J} \leq \lim_{J\rightarrow\infty} \sum_{j=1}^J \sigma_j(X^J)\sigma_j(Y^J)= \sum_{j=1}^{\infty} \sigma_j(X)\sigma_j(Y),\end{equation}
where the last inequality follows by the dominated convergence theorem. We conclude that \eqref{vn1} holds without restrictions.

For the final part of the theorem, note that it is immediate that equality in \eqref{vn1} holds if $X$ and $Y$ share singular vectors. Suppose now that this is not the case, but that equality in \eqref{vn1} holds anyway. For simplicity of notation set $\xi_j=\sigma_j(X)$. Let $J$ and $K$ be integers such that $$\xi_{J-1}<\xi_{J}=\xi_{K}<\xi_{K+1}$$ and define, for $\xi_{K+1}\leq s\leq \xi_{J-1}$, \begin{equation}\label{mui}X(s)=\sum_{j<J}\xi_ju_{X,j}\otimes v_{X,j}+\sum_{J\leq j\leq K}s u_{X,j}\otimes v_{X,j}+\sum_{j>K}\xi_j u_{X,j}\otimes v_{X,j}.\end{equation}
Holding $Y$ fixed, it is clear that $\scal{X(s),Y}$ and $\sum_j \sigma_j(X(s))\sigma_j(Y)$ are affine functions of $s$ in the actual interval. Since the former is dominated by the latter and they equal at the interior point $s=\xi_{J}$, it follows that they must be the same (in $\xi_{K+1}\leq s\leq \xi_{J-1}$). In particular we have \begin{equation}\label{vn3}\scal{X(s),Y}= \sum_j \sigma_j\big(X(s)\big)\sigma_j(Y)\end{equation} for $s=\xi_{K+1}$. Consider now the (piecewise affine) extension of \eqref{mui} defined by \begin{equation*}X(s)=\sum_{j<J}\xi_ju_{X,j}\otimes v_{X,j}+\sum_{J\leq j} \min(s,\xi_j) u_{X,j}\otimes v_{X,j}\end{equation*} on $0\leq s\leq \xi_{J-1}$. The earlier argument can now be bootstrapped to conclude that \eqref{vn3} holds for all $0< s\leq \xi_{J-1}$. Upon taking a limit we conclude that \eqref{vn3} holds for $s=0,$ in which case $X(0)$ has finite rank. Repeating the entire argument with $Y$ as the ``variable'' and $X(0)$ as the fixed matrix, we conclude that identity holds in \eqref{vn1} for $X^J=\sum_{j=1}^J \xi_ju_{X,j}\otimes v_{X,j}$ and $Y^{J'}=\sum_{j=1}^{J'} \eta_j u_{Y,j}\otimes v_{Y,j}$ where $\eta_j=\sigma_j(Y)$ and $J'$ is any index such that $\eta_{J'-1}>\eta_{J'}$. By the finite dimensional version of the theorem, there are common singular vectors $\tilde{u}_j$ and $\tilde{v}_j$ such that $X^J=\sum_{j=1}^J \xi_j \tilde u_{j}\otimes \tilde v_{j}$ and $Y^{J'}=\sum_{j=1}^{J'} \eta_j \tilde u_{j}\otimes \tilde v_{j}$. It is easy to see that this contradicts the initial assumption that $X$ and $Y$ do not share singular vectors, and the proof is complete.
\end{proof}

\section{Appendix II; miscellaneous results}
We provide proofs of the two claims in Section \ref{finer}.
\begin{ex}
In $\ell^2(\N)$ the counting functional $\|x\|_0=\#\{k\in\N:~x_k\neq 0\}$ is weakly l.s.c.
\end{ex}
\begin{proof}
We need to check that the preimage of an open interval of the form $(k,\infty)$ is open in the weak topology. Let $x$ be such that $\|x\|_0> k$, let $K>k$ be an integer such that $\|x\|_0\geq K$, and let $j_1<j_2<\ldots<j_K$ be indices for which $x_{j}\neq 0$. Set $$V=\cap_{m=1}^K\{y\in\ell^2(\mathbb{N}):~|y_{j_m}-x_{j_m}|<|x_{j_m}|/2\}$$
which is open in the weak topology. Clearly $x\in V$ and $\|y\|_{0}\geq K$ for all $y\in V$, which proves the claim.
\end{proof}
\begin{ex}
The rank functional is weakly l.s.c.~on $\mathcal{B}_2(\V_1,\V_2)$.
\end{ex}
\begin{proof}
We only focus on the infinite dimensional case. Let $X\in \mathcal{B}_2(\V_1,\V_2)$ have $\rank(X)\geq K$. As before we need to produce an open set $V$ including $X$ such that $\rank(Y)\geq K$ for all $Y\in V$. By the polar decomposition of compact operators \eqref{SVD} we may pick an orthonormal basis $(v_k)_{k=1}^\infty$ for $\V_1$ and an orthonormal sequence $(u_k)_{k=1}^\infty$ such that $$X=\sum_{k=1}^\infty \sigma_k u_k\otimes v_k $$ where $\sigma_k$ are the singular values ordered decreasingly. Clearly $\sigma_K>0$ and $\para{X,u_k\otimes v_j}=\sigma_k\delta_0{(j-k)}$ where $\delta_0$ is the characteristic function of $\{0\}$ on $\R$. Define $$V=\cap_{j,k=1}^K\{Y\in\mathcal{B}_2(\V_1,\V_2):~|\para{Y,u_k\otimes v_j}-\para{X,u_k\otimes v_j}|<\epsilon\},$$ where $\epsilon$ will be determined later. Let $\iota_1$ and $\iota_2$ be the canonical inclusions (i.e. the operator that sends a vector into itself) from $\mathsf{Span}\{u_k\}_{k=1}^K$ and $\mathsf{Span}\{v_k\}_{k=1}^K$ into $\V_1$ and $\V_2$ respectively, and note that $\iota_2^*$ then acts as the orthogonal projection onto $\mathsf{Span}\{v_k\}_{k=1}^K$. Pick $Y\in V$ and note that \begin{equation}\label{rank}\rank(\iota^*_2 Y \iota_1)\leq \rank(Y).\end{equation}
Moreover $\iota^*_2 Y \iota_1$ is an operator on finite dimensional spaces which in the bases $\{u_k\}_{k=1}^K$ and $\{v_k\}_{k=1}^K$ have the matrix representation $$(y_{j,k})_{j,k=1}^K=(\para{u_j,Y v_k})_{j,k=1}^K=(\para{Y ,u_j\otimes v_k})_{j,k=1}^K.$$ It follows that, upon choosing $\epsilon$ sufficiently small, we can make this matrix arbitrarily close (pointwise) to the diagonal matrix with $\sigma_k$ on the diagonal. By basic linear algebra it follows that $\rank(\iota^*_2 Y \iota_1)=K$ for $\epsilon$ sufficiently small, independent of $Y$. Combining this with \eqref{rank} we conclude that $\rank (Y)\geq K$, which was to be shown.
\end{proof}

\bibliographystyle{plainnat}
\bibliography{referenserLL}

\end{document}